\documentclass{amsart}
\usepackage{tikz}
\usepackage{xcolor}
\usepackage{amssymb,latexsym,amsmath,extarrows}
\usepackage{graphicx,mathrsfs,comment}
\usepackage{hyperref,url}
\usepackage{pict2e}

\usepackage{amstext}
\usepackage{bbm}

\numberwithin{equation}{section}

\usepackage{esint}

\newtheorem{theorem}{Theorem}[section]
\newtheorem{lemma}[theorem]{Lemma}

\newtheorem{proposition}[theorem]{Proposition}
\newtheorem{remark}[theorem]{Remark}

\newtheorem{definition}[theorem]{Definition}
\newtheorem{corollary}[theorem]{Corollary}

\newtheorem{algorithm}{Algorithm}

\makeatletter
\newcommand{\barredsum}{%
  \DOTSB\mathop{\mathpalette\@barredsum\relax}\slimits@
}
\newcommand{\@barredsum}[2]{%
  \begingroup
  \sbox\z@{$#1\sum$}%
  \setlength{\unitlength}{\dimexpr2pt+\ht\z@+\dp\z@\relax}%
  \@barredsumthickness{#1}%
  \vphantom{\@barredsumbar}%
  \ooalign{$\m@th#1\sum$\cr\hidewidth$#1\@barredsumbar$\hidewidth\cr}%
  \endgroup
}
\newcommand{\@barredsumbar}{%
  \vcenter{\hbox{\begin{picture}(0,1)\roundcap\Line(0,0)(0,1)\end{picture}}}%
}
\newcommand{\@barredsumthickness}[1]{
  \linethickness{%
    1.25\fontdimen8
      \ifx#1\displaystyle\textfont\else
      \ifx#1\textstyle\textfont\else
      \ifx#1\scriptstyle\scriptfont\else
      \scriptscriptfont\fi\fi\fi 3
  }%
}
\makeatother

\newcommand{\al}{\alpha}
\newcommand{\be}{\beta}
\newcommand{\ga}{\gamma}
\newcommand{\Ga}{\Gamma}
\newcommand{\de}{\delta}

\newcommand{\De}{\Delta}
\newcommand{\e}{\varepsilon}

\newcommand{\la}{\lambda}

\newcommand{\si}{\sigma}
\newcommand{\Si}{\Sigma}
\newcommand{\vp}{\varphi}
\newcommand{\om}{\omega}

\newcommand{\cs}{\mathcal S}

\newcommand{\cp}{\mathcal P}

\newcommand{\co}{\mathcal O}

\newcommand{\cb}{\mathcal B}

\newcommand{\wt}{\widetilde}
\newcommand{\wh}{\widehat}

\newcommand{\ZR}{\mathbb{R}}
\newcommand{\ZT}{\mathbb{T}}

\newcommand{\ZC}{\mathbb{C}}

\newcommand{\ZN}{\mathbb{N}}

\newcommand{\ZS}{\mathbb{S}}

\newcommand{\Id}{{\textit{1}}}

\newcommand{\bY}{{\bf Y}}

\newcommand{\Tau}{\mathcal{T}}

\newcommand{\cB}{{\mathcal B}}

\newcommand{\R}{\mathbb{R}}

\newcommand{\T}{\mathbb{T}  }

\newcommand{\cO}{\mathcal{O}}

\newcommand{\Z}{\mathbb Z}
\newcommand{\fB}{\mathfrak{B}}

\newcommand{\Sq}{{\rm{Sq}}}

\newcommand{\vf}{\boldsymbol f}
\newcommand{\vg}{\boldsymbol g}

\newcommand{\rap}{{\rm RapDec}}
\newcommand{\dist}{{\rm dist}}

\newcommand{\poly}{{\rm{Poly}}}

\newcommand{\BL}{\textup{BL}}
\newcommand{\bZ}{{\bf Z}}

\newcommand{\Dec}{\textup{Dec}}
\newcommand{\supp}{{\rm supp}}
\newcommand{\proj}{{\rm{Proj}}}

\newcommand{\dc}{\delta_{\circ}}

\begin{document}

\title{New bounds for Stein's square functions in higher dimensions}

\author{Shengwen Gan} \address{Shengwen Gan\\  Deparment of Mathematics, Massachusetts Institute of Technology, USA}\email{shengwen@mit.edu}

\author{Changkeun Oh} \address{Changkeun Oh\\Department of Mathematics\\ University of Wisconsin-Madison, USA}
\email{coh28@wisc.edu}

\author{Shukun Wu} \address{ Shukun Wu\\  Department of Mathematics\\ California Institute of Technology, USA} \email{skwu@caltech.edu}

\date{}
\begin{abstract}
We improve the $L^p(\ZR^n)$ bounds on Stein's square function to the best known range of the Fourier restriction problem when $n\geq4$. Applications including certain local smoothing estimates are also discussed. 
\end{abstract}

\maketitle

\section{Introduction} 
Recall the Bochner-Riesz operator of order $\lambda$
\begin{equation}
    T_t^{\lambda}f(x):=\int_{\R^n}\Big(1-\frac{|\xi|^2}{t^2}\Big)_{+}^{\lambda}\widehat{f}(\xi)e^{ix \cdot \xi}\,d\xi
\end{equation}
and Stein's square function of order $\la$
\begin{equation}
    G^{\lambda}f(x):=\Big( \int_{0}^{\infty}\Big|\frac{\partial}{\partial t}T_t^{\lambda}f(x)\Big|^2\,tdt\Big)^{\frac12}.
\end{equation}
Let $p_n$ be the smallest exponent obtained in \cite{hickman2020note}, which is also the best exponent that the restriction conjecture has been verified when $n\geq4$. Our main result is

\begin{theorem}
\label{thm1.1}
For $p >p_n$ and $n\geq3$, we have
\begin{equation}
    \|G^{\lambda}f\|_{L^p(\R^n)} \leq C_{p,n}\|f\|_{L^p(\R^n)}, \;\;\;\; \mathrm{for} \;\;\; \lambda > \frac{n}{2}-\frac{n}{p}.
\end{equation}
\end{theorem}

The square function $G^\la f$ was introduced by Stein in \cite{Stein-BR}, aiming to study the almost everywhere convergence of the Bochner-Riesz mean. An overview of the square function can be found in  \cite{MR3329855}. Regarding to its $L^p$ problem, it is conjectured that $\|G^\la f\|_p\leq C_{p,n} \|f\|_p$ for $\la>\max\{1/2,n/2-p/2\}$ and $p\geq2$.  Indeed, this conjecture was verified by Carbery \cite{Carbery-MBR} when $n=2$. For higher dimensions there are only partial results. See for instance \cite{gan2021new} for $n=3$ and \cite{Lee-Rogers-Seeger}, \cite{Lee-Sqfcn}  for higher dimensions. Theorem \ref{thm1.1} improves the range when $n\geq4$. We remark that the methods in this paper also give a slightly different proof for the 3-dimensional result in \cite{gan2021new}.

\subsection{Applications}

There are many connections between Stein's square function and other problems. We list some applications below to motivate its $L^p$ estimate. Other than that, Stein's square function or its variation appeared in several places, for example, in the study of variation operators (see \cite{MR4149067}, \cite{beltran2020variation}).

\subsubsection{Radial and maximal radial multipliers}

Let $m$ be a bounded function on $\R$. Define the associated radial multiplier operator $T_mf$  by
\begin{equation}
    \widehat{T_mf}(\xi):=m(|\xi|)\widehat{f}(\xi).
\end{equation}
It is known in \cite{MR747595} that the $L^p$-boundedness of radial multipliers with certain regularity can be obtained by Stein's square function estimate. 

\begin{corollary}
\label{cor-radial}
Let $n \geq 3$ and $\psi \not\equiv 0$ be a compactly supported smooth function on $(0,\infty)$. For every $p \geq p_n$ and 
$ \alpha>n|\frac12-\frac1p|$, we have 
\begin{equation}
\label{radial-esti}
    \|T_mf\|_{L^p(\R^n)} \lesssim \sup_{t>0}\|\psi m(t \cdot)\|_{L^2_{\alpha}(\R)}\|f\|_{L^p(\R^n)}.
\end{equation}
\end{corollary}
\noindent An analogous result for maximal radial multiplier operator also follows from the $L^p$ estimate of Stein's square function. This was shown in \cite{MR848141} Theorem 4. As a special case, we improve the maximal Bochner-Riesz estimate to the same range of $p$. We refer the readers to \cite{MR2784663} and \cite{MR2755678} for the study on general radial multipliers and their applications.

\subsubsection{Dimension of divergent set of the Bochner-Riesz mean}
The almost everywhere convergence of the Bochner-Riesz mean on $L^p$ is established by \cite{MR972135} when $p \geq 2$. Here we consider the pointwise convergence of the Bochner-Riesz mean in Sobolev space.
Consider the dimension of the divergent set.
\begin{equation}
    D_{n,\lambda,p}(\beta):=\sup\limits_{f \in L^p_{\beta}(\R^n)}\mathrm{dim}\{x \in \R^n: \lim_{t \rightarrow \infty} T_t^{\lambda}f(x) \neq f(x) \}
\end{equation}
where dim stands for the Hausdorff dimension.

\begin{corollary}
Let $n\geq 3$, $\lambda>n|\frac12-\frac1p|-\frac12$, and $0< \beta < \frac{n}{p}$. Then for $p > p_n$
\begin{equation}
    D_{n,\lambda,p}(\beta)=n-\beta p.
\end{equation}
\end{corollary}
\noindent This is a corollary of the $L^p$-boundedness of the maximal Bochner-Riesz operator, which is implied by Corollary \ref{cor-radial} (See the discussion below \eqref{radial-esti}). We refer to \cite{MR3026351} for the proof and more discussions on this topic.

\subsubsection{Regularity problem for fractional Schr\"{o}dinger equations}
Let $\al>0$. Consider the partial differential equation $i\partial_t u +(-\Delta_x)^{\alpha/2}u=0$ with the initial data $u(\cdot,0)=f$. Let us denote the solution by $e^{it(-\Delta)^{\alpha/2}}f:=u(x,t)$. It is proved in \cite{Lee-Rogers-Seeger} Proposition 5.1 that Stein's square function estimate implies some mixed norm estimates for the fractional Schr\"odinger operator.

\begin{corollary}
\label{mixed-norm-cor-1}
Let $n \geq 3$, $\alpha \in (0,\infty)$, and $I$ be a compact interval. For $p>p_n$ and $\frac{s}{\alpha}=n(\frac{1}{2}-\frac{1}{p})-\frac{1}{2}$, we have
\begin{equation}
\label{local-smoothing}
    \Big\|\big( \int_I \Big|e^{it(-\Delta)^{\alpha/2}}f\big|^2\,dt \Big)^{\frac12} \Big\|_{L^p(\R^{n})} \lesssim \|f\|_{L^p_s(\R^n)}.
\end{equation}

\end{corollary}

Corollary \ref{mixed-norm-cor-1} itself in fact implies a larger class of mixed norm estimates:

\begin{corollary}
\label{mixed-norm-cor-2}
Suppose that $n \geq 3$ and $I$ is a compact interval. Then for $\al>0$, $q\geq2$, $p>p_n$ and $\frac{s}{\alpha}=n(\frac{1}{2}-\frac{1}{p})-\frac{1}{q}$, we have
\begin{equation}
\label{local-smoothing-2}
    \Big\|\big( \int_I \Big|e^{it(-\Delta)^{\alpha/2}}f\big|^q\,dt \Big)^{\frac{1}{q}} \Big\|_{L^p(\R^{n})} \lesssim \|f\|_{L^p_s(\R^n)}.
\end{equation}
\end{corollary}
\noindent Note that when $p=q$, \eqref{local-smoothing-2} is a local smoothing estimate. Thus, as a byproduct, we obtain some local smoothing estimates for fractional Schr\"ondinger equation (the case $\al=1$ corresponds to the wave equation). As mentioned in \cite{Lee-Rogers-Seeger}, the exponent $\frac{s}{\alpha}=n(\frac{1}{2}-\frac{1}{p})-\frac{1}{q}$ is sharp unless $\al=1$. We will give a simple proof of Corollary \ref{mixed-norm-cor-2} using Corollary \ref{mixed-norm-cor-1} in Appendix C.

\subsubsection{Bilinear Bochner-Riesz}
Consider the bilinear Bochner-Riesz operator \begin{equation}
    B^{\alpha}(f,g)(x):=\iint_{\R^d \times \R^d} e^{2\pi i x \cdot (\xi+\eta)}(1-|\xi|^2-|\eta|^2)^{\alpha}_{+}\widehat{f}(\xi)\widehat{g}(\eta)\,d\xi d\eta.
\end{equation}
The bilinear Bochner-Riesz problem posed by \cite{MR3421992} is to find all possible triples $(p,q,r)$ satisfying
\begin{equation}
    \|B^{\alpha}(f,g)\|_{L^r} \lesssim \|f\|_{L^p(\R^d)}\|g\|_{L^q(\R^d)}.
\end{equation}
It is proved in \cite{MR3856823} Theorem 1.2 that the $L^p$ estimate of Stein's square function implies the boundedness of the bilinear Bochner-Riesz operator for certain range. Similar implication is also true for the maximal bilinear Bochner-Riesz operator, as shown in Theorem 1.2 of \cite{MR4103874}. By their results, Theorem \ref{thm1.1} improves the previously known bounds for the bilinear and maximal bilinear Bochner-Riesz operators. We do not state the ranges here.

\subsection{A sketch of the main ideas}

One of the main tools we use is polynomial partitioning. It was developed by Guth and Katz in \cite{Guth-Katz}, and was introduced to the study of oscillatory integral operators by Guth in \cite{Guth-R3} and \cite{Guth-II}. In there, Guth considered the interpolated norm $\|f\|_2^\al\|f\|_\infty^{1-\al}$ (see \cite{Guth-R3} page 373, 375) and used it to bound $\|Ef\|_p$, where $E$ is the extension operator for paraboloid. Another main tool is the polynomial Wolff axiom. It roughly says that there can not be too many tubes in a thin neighborhood of a variety. The polynomial Wolff axiom was first formulated by Guth and Zahl in \cite{guth2018polynomial}, and later proved by Katz and Rogers in \cite{katz2018polynomial}. Furthermore, a refined version called the nested polynomial Wolff axiom was verified independently in \cite{hickman2019improved} and \cite{zahl2021new}. By combining these two tools, progresses have been made in the Kakeya maximal operator conjecture (see \cite{hickman2019improved}, \cite{zahl2021new}), and in the Fourier restriction problem (see \cite{hickman2018improved}, \cite{hickman2020note}).

Using similar ideas, there are also some progresses related to the Bochner-Riesz operator (See \cite{Wu}, \cite{GOWWZ} and \cite{gan2021new}). Apart from the polynomial methods aforementioned, one new ingredient in \cite{Wu} and \cite{gan2021new} is a refined $L^2$ estimate, and one new ingredient in \cite{GOWWZ} is an application of a pseudo-conformal transform. Nevertheless, it seems to the authors that none of them is a good method to attack higher dimensional Stein's square function. In fact, the refined $L^2$ estimate does provide some improvements upon the square function, but what it provides is weaker than Theorem \ref{thm1.1}. As for the method in \cite{GOWWZ}, their framework relies on Carleson-Sjolin's reduction: one freezes a variable and reduces the Bochner-Riesz problem to some restriction type estimate. However, for Stein's square function, freezing one variable looks very wasteful. Indeed, if we freeze one variable, then the desired estimate is no longer true.

In this paper, we develop a way to use the idea of interpolated norm on Stein's square function. This help us to improve the $L^p$ bounds of the square function to the best known range of the restriction problem. Also, our argument automatically gives a slightly different proof for the results in \cite{Wu} and \cite{GOWWZ}, regarding to the Bochner-Riesz conjecture. Using a similar idea, the authors are able to make improvements on the local smoothing problem for some fractional Schr\"{o}dinger equations. This is discussed in a separate paper \cite{GOW}. 

\smallskip
The rest of the introduction is devoted to an intuitive sketch of our main idea on using the interpolated norm. To save us from abundant notations, let us take $\ZR^3$ as an example. For all higher dimensions the idea is very similar.

Similar to \cite{gan2021new}, we follow Guth's idea to reduce our problem to an $L^p$ estimate for certain broad norm of the vector-valued function defined in \eqref{bdsq}, namely, $\| \Sq \vf\|_{\BL^p_{2,A}(B_R)}^p$. Here  $\vf=\{f_1,\cdots,f_R\}$, $f_j:= m_j\ast f$, and $\wh {m}_j$ is a smooth function supported in an $R^{-1}$ neighborhood of the sphere of radius $1+j/R$. The problem then boils down to the estimate
\begin{equation}\label{sk0}
    \| \Sq \vf\|_{\BL^p_{2,A}(B_R)}^p\lesssim R^{p-3}\|f\|_{L^2(B_R)}^2\|f\|_{L^\infty(B_R)}^{p-2}.
\end{equation} 
We iteratively apply the polynomial partitioning to break down the broad norm. When the iteration stops, we obtain a collection of cells $\co=\{O\}$ at scale $\rho$, associated functions $\{\vf_{O}\}_{O\in\co}$, and a broad estimate
\begin{equation}
\label{sk1}
    \|\Sq\vf\|_{\BL^p_{2,A}(B_R)}^p\lesssim \sum_{O\in\co}\|\Sq\vf_{O}\|_{\BL^p_{2,A}(B_R)}^p,
\end{equation}

Here is our key estimate: For each cell $O$ there is a small set $X\subset B_R$ such that
\begin{equation}
\label{refined-esti-intro}
    \| \Sq\vf_{O} \|_{2}^2\lesssim \rho R^{-1}\|f\Id_X\|_2^2 \lesssim\rho R^{-1}|X|\cdot\|f\|_{L^\infty(B_R)}^2.
\end{equation}
To get \eqref{refined-esti-intro}, a crucial observation is that only a very small portion of $f$, which we denote by $f\Id_X$, would make contribution to the function $\Sq\vf_{O}$. The reason is that $\Sq\vf_{O}$ is concentrated on some tangent wave packets, and the set $X$ can be think of as some ``thin-neighborhood" of the variety related to $O$. One may want to compare \eqref{refined-esti-intro} to the estimate on $\int |f_{\tau,j,tang}|^2$ in \cite{Guth-R3} page 402. The detailed proof of \eqref{refined-esti-intro} is presented in Section \ref{section8}.

\medskip

\noindent {\bf Organization of the paper.} In Section 2, we give the general statement of our results. We make several reduction of the target operator in Section 3. Section 4 is devoted to the wave packet decomposition. Section 5,6,7 contain the iteration of polynomial partitioning. Our key estimate is given in Section 8, and in the same section we finish the proof. There are also three appendices, and the proof of Corollary \ref{mixed-norm-cor-2} is included in Appendix C.

\medskip

\noindent {\bf Notation.}

\noindent$\bullet$  We use $a\sim b$ to denote $ca\leq b\leq Ca$ for some unimportant constants $c$ and $C$. We use $a\lesssim b$ to denote $a\leq Cb$ for an unimportant constant $C$. These constants may change from line to line.

\noindent$\bullet$ We write $A(R) \leq \mathrm{RapDec}(R)B$ to mean that for any power $\beta$, there is a constant $C_{N}$ such that
\begin{equation}
\nonumber
    A(R) \leq C_{N}R^{-N}B \;\; \text{for all $R \geq 1$}.
\end{equation}

\noindent$\bullet$ For every number $R>0$ and set $S$, we denote by $N_{R}(S)$ the $R$-neighborhood of the set $S$.

\noindent$\bullet$ The symbol $B^n(x,r)$ represents the open ball centered at $x$, of radius $r$, in $\ZR^n$. We sometimes simply write it as $B_r$ when we only care about its radius but not its position.

\noindent$\bullet$ We use $\Id_X$ to denote the characteristic function of the set $X\subset\ZR^n$.

\noindent$\bullet$ For $X$ being a ball $B_r$ or a rectangle, we use $\om_{X}$ to denote the indicator function of $B_r$ with rapidly decaying term. For example, $w_{B_R}(x)=(1+\dist(x,B_r))^{-100n}$.

\noindent$\bullet$ By saying an $r$-tube, we mean a rectangle of dimensions $R^\beta r^{1/2}\times\cdots\times R^{\beta}r^{1/2}\times r$. Here $\beta$ is a tiny number to be defined later.

\noindent$\bullet$ By saying an $r$-cap, we mean a cap of radius $r$ in $\ZS^{n-1}$.

\noindent$\bullet$ For any cap $\tau\subset \ZS^{n-1}$, we use $c_\tau$ to denote the center of $\tau$.

\medskip

\noindent
{\bf {Some numbers.}} 
We will encounter many different numbers in the paper. For readers' convenience, we summarize all of them here. Note that here we only give a rough description of these numbers. The precise definition will be given later in the paper.

\noindent$\bullet$ $\be\ll\dc\ll\de\ll\e$, where $\e$ is a very small number, $\de=\e^2$, $\de_\circ=\de^3=\e^6$ and $\be=\e^{1000}$. We remark that $\beta$ is used to handle the rapidly decaying tail of wave packets.

\noindent$\bullet$ $K\sim R^{\e^{10000}}$ is the number appearing in the broad-narrow argument.

\noindent$\bullet$ $A\sim \log\log R$ is the number in the subscript of the broad norm $\|f\|_{\BL_{k,A}^p} $. Also, we would like $A$ to be of form $2^{\ZN}$.

\noindent$\bullet$ $d=d(\e)$ is a big number (depending on $\e$) which is the degree of the polynomial in the polynomial partitioning argument.

\noindent$\bullet$ In this paper, the implicit constant in ``$\lesssim$" may depend on $\de$, but never depend on $R$. We also define the notation ``$\lessapprox$" in Remark \ref{pretend}.

\medskip

\noindent {\bf Acknowledgement.} The authors would like to thank Shaoming Guo, Larry Guth, Xiaochun Li and Andreas Seeger for valuable discussions. C.O. and S.W. would also like to thank Hong Wang and Ruixiang Zhang for  valuable discussions during the collaboration of the paper \cite{GOWWZ}. C.O. was partially supported by the NSF grant DMS-1800274.

\section{General statement}
    
We will state and prove a more general result. To do this, let us first introduce a class of hypersurfaces $\{\Ga_j\}_{1\le j\le R}$. Here, we may assume $R$ is an integer for convenience. Let $\Phi(\bar\xi,t): B^{n-1}(0,1) \times [1,2] \rightarrow \R$ be a smooth function satisfying the following two properties:

\begin{itemize}
    \item $|\partial_t\Phi(\bar{\xi},t)| \sim 1, \;\; |\partial_t \nabla_{\bar{\xi}}\Phi(\bar{\xi},t)| \lesssim 1$ for all $t \in [1,2]$ and $\bar{\xi} \in B^{n-1}(0,1)$.
    \item All the eigenvalues of the Hessian matrix $\nabla_{\bar{\xi}}^2\Phi$ have the same sign and magnitude $\sim 1$, uniformly in $t$.
\end{itemize}
Fix a collection of $R^{-1}$-separated points $\{t_j \}_{1 \leq j \leq R}$ lying in $[1,2]$ (one could choose $t_j=1+\frac{j}{R}$). For each $t_j$, we consider the hypersurface
\begin{equation}\label{surface}
    \Gamma_j:=\{(\bar{\xi},\Phi(\bar{\xi};t_j)):|\bar{\xi}| \leq 1/2  \}.
\end{equation}

\begin{remark}
\rm

In the paper the subscript $j$ is only used for these hypersurfaces.
\end{remark}

We introduce a square function associated to these hypersurfaces. For each $j$, let $\eta_j\!:\!\R^{n-1}\to\ZR$ and let $\psi_j\!:\!\R\to\ZR$ be two smooth functions, with $\supp(\eta_j)\subset B^{n-1}(0,1/2)$ and $\supp(\psi_j)\subset [-1/2,1/2]$. Let $\phi:\ZR^n\times[0,1]\to\ZR$ be a smooth function that takes the value in $[1/2,2]$. All these three functions have bounded derivatives of any order, independent of $j$. Define a general Fourier multiplier
\begin{equation}
\label{kernel-1}
    \wh{m}_j(\xi):=\eta_j(\bar\xi)\psi_j\Big(\frac{\xi_n-\Phi(\bar\xi;t_j)}{R^{-1}}\phi(\xi;t_j)\Big).
\end{equation}
Hence $\textup{supp}(\wh m_j)\subset N_{R^{-1}}\Ga_j$. We also define the associated operator 
\begin{equation}
\label{operator-1}
    S_j f(x):={m}_j\ast f(x).
\end{equation}
The discrete square function associated to the hypersurfaces $\{\Ga_j\}_{j}$ is defined as:

\begin{definition}[Square function]
\label{defsqfunction}
For any Schwartz function $f$, define
\begin{equation}
\label{sqfcn-1}
    \Sq f(x):=\Big(\sum_{j=1}^R|S_jf(x)|^2\Big)^{1/2}.
\end{equation}
\end{definition}

We have the following reduction:
\begin{lemma}\label{discretelem}
Let $\Sq f$ be defined in \eqref{sqfcn-1}. For some fixed $p_0\geq \frac{2n}{n-1}$, suppose that for any $\e>0$ we have
\begin{equation}
\label{discrete-sqfcn}
    \|\Sq f\|_{L^{p_0}(\R^n)}\leq C_\e R^{\frac{n-1}{2}-\frac{n}{p_0}+\e}\|f\|_{L^{p_0}(\R^n)}.
\end{equation}
Then Theorem \ref{thm1.1} holds for $p>p_0$.
\end{lemma}

Lemma \ref{discretelem} was proved in \cite{gan2021new} when $n=3$. The readers may look at how Lemma 2.7 implies Theorem 1.1 in \cite{gan2021new}. For $n\geq 4$, this lemma can be proved in the same way.

The wanted estimate \eqref{discrete-sqfcn} boils down to a local estimate stated below.

\begin{theorem}\label{sqfcn-localest}
Assuming the same notations as in Lemma \ref{discretelem} and $p\ge p_n$, then for any $\e>0$ we have
\begin{equation}
\label{discrete-sqfcn2}
    \|\Sq f\|_{L^p(B_R)}\leq C_\e R^{\frac{n-1}{2}-\frac{n}{p}+\e}\|f\|^{\frac{2}{p}}_{L^2(w_{B_R})}\|f\|^{1-\frac{2}{p}}_{L^\infty(\R^n)}.
\end{equation}
\end{theorem}

\begin{proof}[Proof of Theorem \ref{sqfcn-localest} implying Theorem \ref{thm1.1}]

If Theorem \ref{sqfcn-localest} holds, then by taking the $p$-th power to both side of \eqref{discrete-sqfcn2} and then summing over $B_R\subset \R^n$ we get 
\begin{align*}
    \|\Sq f\|_{L^p(\R^n)}^p&\leq C_{\e} R^{(\frac{n-1}{2}-\frac{n}{p}+\e)p}\sum_{B_R}\|f\|^{2}_{L^2(w_{B_R})}\|f\|^{p-2}_{L^\infty(\ZR^n)}\\
    &\leq C_{\e} R^{(\frac{n-1}{2}-\frac{n}{p}+\e)p}\|f\|^{2}_{L^2(\R^n)}\|f\|^{p-2}_{L^\infty(\R^n)}.
 \end{align*} 
Via real interpolation (See Theorem 1.4.19 in \cite{Grafakos-249} for instance), we obtain \eqref{discrete-sqfcn}, and hence by Lemma \ref{discretelem} we prove Theorem \ref{thm1.1}.
\end{proof}

Now we have reduced Theorem \ref{thm1.1} to Theorem \ref{sqfcn-localest}.
In the next section, we will introduce some more notations and preliminary lemmas, and do broad-narrow reduction to further reduce Theorem \ref{sqfcn-localest} to some broad norm estimate.

\section{Reductions}\label{sectionbn}
In this section, we introduce the notion of vector-valued function and its broad norm. We will see that Theorem \ref{sqfcn-localest} is reduced to the broad norm estimate for the vector-valued function.

\subsection{Gauss maps}

We first give some definitions.

\begin{definition}[Gauss map]
Let $\Phi$ and $\{\Gamma_j\}$ be given in \eqref{surface}. For each $j$, we define the Gauss map of the hypersurface $\Gamma_j$ as:
\begin{equation}
\label{Gauss-map}
    G_j: B^{n-1}(0,1/2)\rightarrow \ZS^{n-1},\ \ \
    G_j(\bar\xi)=\frac{(\nabla_{\bar\xi}\Phi(\bar\xi;t_j),1)}{|(\nabla_{\bar\xi}\Phi(\bar\xi;t_j),1)|}.
\end{equation}
Since there is a one-to-one correspondence between $B^{n-1}(0,1/2)$ and $\Ga_j$, we can unambiguously think of $G_j$ as a function from $\Ga_j$ to $\ZS^{n-1}$.
\end{definition}

\begin{definition}
For any cap $\si\subset \ZS^{n-1}$, we define 
\begin{equation}\label{defslabsurface}
    \Ga_j(\si):=\{ \xi\in\Ga_j: G_j(\xi)\in \si \}.
\end{equation}
\end{definition}

What follows are several lemmas about the Gauss map $G_j$ that are not hard to show. We omit their proofs.
\begin{lemma}
\label{Gauss-map-lem}
For any $1\leq j\leq R$, the Gauss map $G_j(\bar\xi)$ is smooth and injective. In particular, when $\bar\xi\in B^{n-1}(0,1/2)$, one has
\begin{equation}
\label{derivative-Gauss}
    |\partial^\al G_j|\leq C_\al,\hspace{5mm}\al\in\ZN^{n-1}
\end{equation}
uniformly for all $G_j$.
\end{lemma}

\begin{lemma}
\label{Gauss-map-spt}
Suppose that $\si\subset  \ZS^{n-1}$ is a $\rho^{-1/2}$-cap with $1\le \rho\le R$. Then for any $1\leq j\leq R$, the set $G_j^{-1}(\si)$ is morally a $\rho^{-1/2}$-ball, which means there is a $\rho^{-1/2}$-ball $B$ in $B^{n-1}(0,1/2)$ such that
\begin{equation}
    cB\subset G_j^{-1}(\si)\subset CB.
\end{equation}
for two absolute constants $c<1$ and $C>1$.

As a result, we see that $\Ga_j(\si)$ is a $\rho^{-1/2}$-cap on $\Ga_j$.
\end{lemma}

\subsection{Vector-valued functions}

For technical reasons, it's convenient  to work on vector-valued functions. In this subsection, we discuss the properties for vector-valued functions.

\begin{definition}[Square function for vector-valued functions]\label{defsqfunc}
Given a vector-valued function $\vg=\{g_1,\cdots,g_R\}$, we define the square function of $\vg$ by:
\begin{equation}\label{sqsq}
    \Sq\vg(x):=\Big(\sum_{j=1}^R|g_j(x)|^2\Big)^{1/2}.
\end{equation}
\end{definition}

Recall \eqref{operator-1}, \eqref{sqfcn-1}. Under these notations, if we define $f_j$ and $\vf$ as
\begin{equation}
\label{vector-f}
    f_j:=S_{j}f,\hspace{1cm}\vf:=\{f_1,\ldots,f_R\},
\end{equation}
then $\Sq f=\Sq\vf$. Hence we see that \eqref{discrete-sqfcn2} is equivalent to
\begin{equation}\tag{$\Sq_p$}\label{localsq}
    \|\Sq \vf\|_{L^p(B_R)}\leq C_\e R^{\frac{n-1}{2}-\frac{n}{p}+\e}\|f\|^{\frac{2}{p}}_{L^2(w_{B_R})}\|f\|^{1-\frac{2}{p}}_{L^\infty(\R^n)}.
\end{equation}

\medskip

Fix a number $\rho\ge 1$. In the following discussion, assume that we are given a vector-valued function $\vg=\{g_1,\cdots,g_R\}$ such that each component satisfies $\textup{supp}(\wh g_j)\subset N_{R^{\beta}\rho^{-1}}\Ga_j$. 
Here, we remark that $\beta=\e^{1000}$ is a tiny number aiming to handle rigorously the rapidly decaying tail of wave packets. We recommend readers assuming $\beta=0$ in the first reading.

For each $\rho^{-1/2}$-cap $\tau\subset \ZS^{n-1}$, we want to define
$$ \vg_\tau=\{g_{1,\tau},\cdots,g_{R,\tau}\}, $$
so that each component, roughly speaking, satisfies
$$ \textup{supp}(\wh g_{j,\tau})\subset N_{R^{\beta} \rho^{-1}}\Ga_j(2\tau). $$

We fix a scale $\rho$ and let $\{\tau\}$ be a set of $\rho^{-1/2}$-caps that form a finitely overlapping cover of $\ZS^{n-1}$. Then for each $1\le j\le R$, $\{\Ga_j(\tau)\}$ form a finitely overlapping cover of $\Ga_j$.
Let $\{\wh\vp_{j,\tau}(\xi)\}$ be a partition of unity associated to $N_{R^{\beta}\rho^{-1}}\Ga_j$ so that each $\wh\vp_{j,\tau}$ is a smooth cut off at the slab $N_{R^{\beta}\rho^{-1}}\Ga_j(\tau)$ of dimensions $\rho^{-1/2}\times\cdots\times\rho^{-1/2}\times R^\beta \rho^{-1}$. In this way, $\sum_\tau \wh\vp_{j,\tau}$ forms a smooth approximation to $\Id_{N_{R^{\beta}\rho^{-1}}\Ga_j}$, and $\sum_\tau \wh\vp_{j,\tau}=1$ on $N_{R^{\beta}\rho^{-1}}\Ga_j$.
We define
\begin{equation}\label{gtau}
    g_{j,\tau}:=\vp_{j,\tau}*g_j,\ \ \vg_\tau:=\{g_{1,\tau},\cdots,g_{R,\tau}\}.
\end{equation}
From the condition on the Fourier support of each $g_j$, we have
$$ g_j=\sum_\tau g_{j,\tau}. $$
We define the square function associated to $\tau$ as
$$ \Sq \vg_{\tau}:=\Big(\sum_{j=1}^R|g_{j,\tau}(x)|^2\Big)^{1/2}. $$
One can check $\Sq\vg_{\ZS^{n-1}}=\Sq\vg$.

\subsection{Broad norm}

The broad norm we define here is a bit different from the usual way. We first fix an $M^{-1}$-cap $\si\subset \ZS^{n-1}$ and assume that each component of our function $\vg$ has Fourier support in the cap whose direction is determined by $\si$.

\begin{definition}[Broad norm]\label{defbroad}
Suppose that $\si\subset \ZS^{n-1}$ is an $M^{-1}$-cap and $\si$ is covered by finitely overlapping $K^{-1}M^{-1}$-caps: $\si=\cup\tau$. Let $\vg=\{g_1,\cdots,g_R\}$ be a vector-valued function satisfying $\textup{supp}(\wh g_{j})\subset N_{(KM)^{-2}}\Ga_j(\si)$. 

For any linear subspace, $V\subset \R^n$, we use $\angle(\tau,V)$ to denote the smallest angle between any two vectors in $\tau$ and $V$. We write $\tau \notin V$ to mean $\angle(\tau,V)\ge K^{-1}$. Otherwise, we write $\tau\in V$.

Now we partition $\R^n$ into rectangles of dimensions $MK^2\times \cdots\times MK^2\times M^2K^2$, pointing to the direction $c_\si$ (here we use $c_\si$ to denote the center of the cap $\si$). We denote this partition by $\R^n=\sqcup B_{MK^2\times M^2K^2}$.
For each rectangle $B_{MK^2\times M^2K^2}$, we define the $k$-broad norm of $\Sq\vg$ on $B_{MK^2\times M^2K^2}$ by 
\begin{equation}
    \nu(B_{MK^2\times M^2K^2}):=\min_{V_1,\cdots,V_A}\big( \max_{\begin{subarray}{c}\tau\notin V_a\\ \textup{for~any~}1\le a\le A\end{subarray}}\int_{B_{MK^2\times M^2K^2}}|\Sq \vg_\tau|^p \big).
\end{equation}
where the minimum is over $(k-1)$-dimensional subspaces of $\R^n$.

Finally, for $U\subset \R^n$, we define
\begin{equation}\label{bdsq}
    \|\Sq \vg\|_{\BL^p_{k,A}(U)}:=\big( \sum_{B_{MK^2\times M^2K^2}}\frac{|B_{MK^2\times M^2K^2}\cap U|}{|B_{MK^2\times M^2K^2}|}\nu(B_{MK^2\times M^2K^2}) \big)^{1/p}.
\end{equation}
\end{definition}

\begin{remark}
\rm

Our definition of the broad norm depends on $\si$. Even though $\si$ does not appear in the notation, it will be clear form the text when we use it. When we consider the broad norm for $\vg_\si$, it is always with respect to the cap $\si$.
\end{remark}

\begin{remark}
\rm

It may not be clear to readers why we use rectangles $B_{MK^2\times M^2K^2}$ of dimensions $MK^2\times\cdots \times MK^2\times M^2K^2$. To get some intuitions, we first take a look at the case $M=1$. In this case, $B_{MK^2\times M^2K^2}$ is just the $K^2$-ball which was used in \cite{Guth-II}. For general $M^{-1}$-cap $\si$, if we can do parabolic rescaling to transform the cap $\si$ into some $1$-cap, then in the physical space $B_{MK^2\times M^2K^2}$ would be transformed into $B_{K^2}$. This suggests us to use $B_{MK^2\times M^2K^2}$ in Definition \ref{defbroad}.
\end{remark}

We have triangle inequality and H\"older's inequality for the broad norm.

\begin{lemma}[Triangle inequality]
Assuming the same notation, we have
\begin{equation}\label{broad-triangle}
    \| \Sq(\vg_1+\vg_2) \|_{\BL^p_{k,2A}(U)}\lesssim \| \Sq\vg_1 \|_{\BL^p_{k,A}(U)}+\| \Sq\vg_2 \|_{\BL^p_{k,A}(U)} 
\end{equation}
\end{lemma}

\begin{lemma}[H\"older's inequality]
For $1\le p,p_0,p_1<\infty$ and $0\le \al\le 1$ satisfying
$$ \frac{1}{p}=\frac{1-\al}{p_0}+\frac{\al}{p_1}, $$
we have
\begin{equation}\label{broad-holder}
    \| \Sq\vg \|_{\BL^p_{k,2A}(U)}\lesssim \| \Sq\vg \|_{\BL^{p_0}_{k,A}(U)}^{1-\al}\| \Sq\vg \|_{\BL^{p_1}_{k,A}(U)}^{\al} 
\end{equation}
\end{lemma}

Now we can state the main estimate in this paper. The most part of the paper is devoted to the proof of this estimate.
\begin{theorem}[The main result for broad norm]\label{thmbroadnorm}
Fix a $M^{-1}$-cap $\si\subset \ZS^{n-1}$.
Let $2\le k\le n-1$, and 
\begin{equation}\label{defpk}
    p> p_n(k):=2+\frac{6}{2(n-1)+(k-1)\prod_{i=k}^{n-1} \frac{2i}{2i+1}}.
\end{equation}
Recall the definition of $\vf$ in \eqref{vector-f}. Then for every $\epsilon>0$, there exists $A$ such that 
\begin{equation}\tag{$\mathrm{BL}_{k,A}^p$}
    \| \Sq \vf_\si \|_{\BL^p_{k,A}(B_R)}^p\lesssim_\e R^{p\e}R^{\frac{n-1}{2}p-n}M^{2n-(n-1)p}
    \|\Sq \vf_\si\|_{L^2(\om_{B_{R}})}^2
    \|f\|_{L^{\infty}(\R^n)}^{p-2}.
    \label{broadinequality}
\end{equation}
\end{theorem}

We will begin the proof of Theorem \ref{thmbroadnorm} in Section 5. In the rest of this section, we show how it implies \eqref{localsq} and hence Theorem \ref{sqfcn-localest}. Actually, we will prove the following lemma. Then by optimizing the $k$, we see \eqref{discrete-sqfcn2} holds for $p\geq p_n$.

\begin{lemma}[Reduction to the broad norm estimate]\label{reduction} Let $n \geq 3$ and $2 \leq k \leq n-1$.
Suppose that
\begin{equation}
    2+\frac{4}{2n-k} \leq p \leq 2+\frac{2}{k-2}.
\end{equation}
Them the broad norm estimate \eqref{broadinequality} implies the square function estimate \eqref{localsq}.
\end{lemma}

\begin{proof}[Proof of Lemma \ref{reduction}]
We first do a one-step broad-narrow decomposition. Here we need to introduce a new notation. For the set of $K^{-1}$-cap $\Si_1=\{\si_1\}$ that form a finitely overlapping cover of $\ZS^{n-1}$, we want to define the square function whose frequency is restricted to a subcollection of these caps. For a $(k-1)$-plane $V\subset \R^n$, denote the caps that form an angle less than $K^{-1}$ with $V$ by 
\begin{equation}\label{fsi}
    \Si_1(V):=\{ \si_1 \subset \ZS^{n-1} : \si_1\in V \}.
\end{equation}
We define
\begin{equation}\label{setofcap}
    \vf_{\Si_1(V)}(x):= \sum_{\si_1\in\Si_1(V)} \vf_{\si_1},
\end{equation}
so
\begin{equation}
    \Sq\vf_{\Si_1(V)}(x):= \Big( \sum_{j=1}^R \big|\sum_{\si_1\in\Si_1(V)} f_{j,\si_1}(x)\big|^2 \Big)^{1/2}.
\end{equation}
Thus, for each $B_{K^2}$ and any $(k-1)$-planes $V_1\cdots,V_A$, we have
\begin{equation}
    \int_{B_{K^2}}|\Sq \vf|^p\lesssim K^{O(1)}\max_{\si_1\notin V_a, 1\le a\le A}\int_{B_{K^2}}|\Sq\vf_{\si_1}|^p+\sum_{a=1}^A\int_{B_{K^2}}\big|  \Sq\vf_{\Si_1(V_a)} \big|^p.
\end{equation}

Recall the definition of the broad norm in Definition \ref{defbroad} for the case $M=1$. We optimize the choice for $V_1,\cdots, V_A$ to obtain
\begin{equation}\label{redc1}
    \int_{B_{K^2}}|\Sq \vf|^p\lesssim K^{O(1)} \|  \Sq \vf\|_{\BL^p_{k,A}(B_{K^2})}^p+\sum_{a=1}^A\int_{B_{K^2}}\big|  \Sq\vf_{\Si_1(V_a)} \big|^p.
\end{equation}
On the right hand side of \eqref{redc1}, we call the first term broad term and the second term narrow term. In order to deal with the narrow term, we need a decoupling inequality for the square function. 
To state our decoupling lemma, we define $R$ hypersurfaces in $\R^m$, each of which, by an abuse of notation, is also denoted by $\Ga_j$:
\begin{equation}\label{msurface}
    \Ga_j:=\{\xi=(\bar\xi,\xi_m): \xi_m=\Phi_j(\bar\xi)\}.
\end{equation}
Here $\Phi_j:\R^{m-1}\rightarrow \R$ satisfies that $D^2 \Phi_j$ has all eigenvalues lying in $[1/2,2]$.

\begin{lemma}[Decoupling for the square function]\label{lemdec}
Let $\si\subset \ZS^{m-1}$ be a cap of radius $M^{-1}$, and $\Tau_{\sigma}=\{\tau\}$ be a collection of $K^{-1}M^{-1}$-caps that covers $\sigma$. Let $\vg=\{g_1,\cdots,g_R\}$ be any vector-valued function, such that each $g_j:\R^{m}\rightarrow \ZC$ has Fourier support in $N_{K^{-2}M^{-2}}\Ga_j(\si)$  .
Let $B_{MK^2\times M^2K^2}\subset \R^{m}$ be a rectangle of dimensions $MK^2\times\cdots\times MK^2\times M^2K^2$, pointing to the direction $c_\si$. Then for $2 \leq p \leq 2+\frac{2}{m-1}$, we have
\begin{equation}\label{lemdecineq}
    \|\Sq \vg\|_{L^{p}(B_{MK^2\times M^2K^2})}\lesssim_{\e'} (KM)^{\e'}\Big(\sum_{\tau \in \Tau_{\sigma}}\|\Sq \vg_\tau\|^2_{L^{p}(\om_{B_{MK^2\times M^2K^2}})}\Big)^{1/2}.
\end{equation}
\end{lemma}

\begin{remark}
\rm

We have chosen $K=R^{\e^{10000}}$ and note that $M\le R^{1/2}$, so when $\e'$ in \eqref{lemdecineq} is sufficiently small then
$$ (KM)^{\e'}
\le K^{\de/2}. $$
It looks plausible to extend the decoupling inequality to the range $2 \leq p \leq 2(n+1)/(n-1)$ by adapting the argument of Bourgain-Demeter \cite{Bourgain-Demeter}, but we do not pursue it here, since the shorter range $2 \leq p \leq 2n/(n-1)$ is enough for our application.
\end{remark}

The proof of Lemma \ref{lemdec} is included in the Appendix B. As it is done in \cite{Guth-II} Lemma 9.3, we apply Lemma \ref{lemdec} with $M=1$ and $m=k-1$ to each $(k-1)$-dimensional slice of $B_{K^2}$ that is parallel to $V_a$, and then integrate over all the slices. Note the assumption of Lemma \ref{reduction} gives $p\le 2+\frac{2}{k-2}$. We hence obtain
\begin{equation}
    \| \Sq \vf_{\Si_1(V_a)} \|_{L^p(B_{K^2})}\lesssim K^{\de/2} \big(\sum_{\si_1\in\Si_1(V_a)}\|\Sq \vf_{\si_1} \|^2_{L^p(\om_{B_{K^2}})}\big)^{1/2}.
\end{equation}
Note that the number of $\si_1\in \Si_1(V_a)$ is $\lesssim K^{k-2}$. By H\"older's inequality we get
\begin{equation}\label{redc2}
    \int_{B_{K^2}}\big|  \Sq \vf_{\Si_1(V_a)} \big|^p\lesssim K^{\de/2}K^{(k-2)(\frac{p}{2}-1)}\sum_{\si_1\in V_a}\int |\Sq\vf_{\si_1}|^p\om_{B_{K^2}}.
\end{equation}

Plug \eqref{redc2} into \eqref{redc1} and
sum over $B_{K^2}$ so that
\begin{equation}
\nonumber
    \int|\Sq \vf|^p w_{B_R}
    \lesssim K^{O(1)}\|\Sq \vf\|^p_{\BL^p_{k,A}(w_{B_R})}+K^{\de}K^{(k-2)(\frac{p}{2}-1)}\sum_{\si_1\in\Si_1}\int |\Sq\vf_{\si_1}|^p w_{B_R}.
\end{equation}
We apply Theorem \ref{thmbroadnorm} with $M=1$ and $\si=\ZS^{n-1}$ to the first term on the right hand side to obtain
\begin{equation}\label{pfmain}
\begin{split}
    \int|\Sq \vf|^pw_{B_R}&\lesssim K^{O(1)}R^{p \e}R^{\frac{n-1}{2}p-n}
    \|\Sq \vf\|^2_{L^2(\om_{B_{R}})}\|f\|_{L^{\infty}(\R^n)}^{p-2}
    \\&
    +K^{\de}K^{(k-2)(\frac{p}{2}-1)}\sum_{\si_1\in\Si_1}\int |\Sq\vf_{\si_1}|^pw_{B_R}.
\end{split}
\end{equation}

Next, we focus on the second term on the right hand side of \eqref{pfmain}. For each $K^{-1}$-cap $\si_1$, we tile $\R^n=\cup B_{K^{-1}R\times R}$ where each $B_{K^{-1}R\times R}$ is a rectangle of dimensions $K^{-1}R\times \cdots\times K^{-1}R\times R$ whose long side points to the direction $c_{\si_1}$. Note that $$w_{B_R}\lesssim \sum\limits_{B_{K^{-1}R\times R}\subset B_R}{ \om_{B_{K^{-1}R\times R}}}.$$
We repeat the above broad-narrow argument to $\int |\Sq\vf_{\si_1}|^p\om_{B_{K^{-1}R\times R}}$ to obtain
\begin{equation}
\begin{split}
    \int|\Sq \vf_{\si_1}|^p\om_{B_{K^{-1}R\times R}}
    &
    \lesssim K^{O(1)}\|\Sq \vf_{\si_1}\|^p_{\BL^p_{k,A}(\om_{B_{K^{-1}R\times R}})}
    \\&
    +K^{\de}K^{(k-2)(\frac{p}{2}-1)}\sum_{\si_2\subset \si_1}\int |\Sq\vf_{\si_2}|^p\om_{B_{K^{-1}R\times R}}.
\end{split}
\end{equation}
Here $\Si_2=\{\sigma_2\}$ are $K^{-2}$-caps that tile $\ZS^{n-1}$. By summing over the balls $ B_{K^{-1}R\times R}$ and applying Theorem \ref{thmbroadnorm} to the first term on the right hand side above, we obtain
\begin{equation}
\nonumber
\begin{split}
    \int|\Sq \vf_{\si_1}|^p w_{B_R} &\lesssim K^{O(1)}K^{2n-(n-1)p}R^{p\e}R^{\frac{n-1}{2}p-n}
    \|\Sq \vf_{\si_1}\|^2_{L^2(\om_{B_{R}})}\|f\|_{L^{\infty}(\R^n)}^{p-2}
    \\&
    +K^{\de}K^{(k-2)(\frac{p}{2}-1)}\sum_{\si_2\subset \si_1}\int |\Sq\vf_{\si_2}|^p\om_{B_{R}}.
\end{split}
\end{equation}
Plugging into the right hand side of \eqref{pfmain}, we get 
\begin{align*}
    \int&|\Sq \vf|^p\om_{B_{R}}\lesssim K^{O(1)}R^{p \e}R^{\frac{n-1}{2}p-n}
    \|\Sq \vf\|^2_{L^2(\om_{B_{R}})}\|f\|_{L^{\infty}(\R^n)}^{p-2}
    \\&+K^{O(1)}K^{(k-2)(\frac{p}{2}-1)}K^{2n-(n-1)p}R^{p\e}R^{\frac{n-1}{2}p-n}
    \sum_{\sigma_1\in\Si_1}\|\Sq \vf_{\si_1}\|^2_{L^2(\om_{B_{R}})}\|f\|_{L^{\infty}(\R^n)}^{p-2} \\ \nonumber
    &+K^{2\de}K^{2(k-2)(\frac{p}{2}-1)}\sum_{\si_2\in\Si_2}\int |\Sq\vf_{\si_2}|^p\om_{B_{R}}.
\end{align*}
Now we just need to iterate this process for the narrow part. Thus, we obtain
\begin{align}\label{redc5}
    &\int|\Sq \vf|^p\om_{B_{R}}\lesssim K^{m\de}K^{m(k-2)(\frac{p}{2}-1)}\sum_{\si_m\in\Si_m}\int |\Sq\vf_{\si_m}|^p\om_{B_{R}}+ 
    \\ \nonumber
    \sum_{l=0}^{m-1} K^{O(1)}&K^{l(k-2)(\frac{p}{2}-1)}K^{l(2n-(n-1)p)}R^{p\e}R^{\frac{n-1}{2}p-n}
    \sum_{\sigma_l\in\Si_l}\|\Sq \vf_{\si_l}\|^2_{L^2(\om_{B_{R}})}\|f\|_{L^{\infty}(\R^n)}^{p-2}.
\end{align}
Here $m$ is the integer satisfying $K^m\sim R^{1/2}$, each $\Si_l=\{\si_l\}$ is a collection of $K^{-l}$-caps that tile $\ZS^{n-1}$. 

For the second term on the right hand side of \eqref{redc5}, by the assumption that $2+\frac{4}{2n-k} \leq p$ in Lemma \ref{reduction}, we have the factor of $K$ satisfies
\begin{equation}
\nonumber
     K^{O(1)}K^{l(k-2)(\frac{p}{2}-1)}K^{l(2n-(n-1)p)}\le K^{O(1)}\lesssim R^{\de}. 
\end{equation}
As for the first term on the right hand side of \eqref{redc5}, we note by the assumption that $2+\frac{4}{2n-k}\le p$, the factor of $K$ has bound
\begin{equation}
\nonumber
    K^{m\de}K^{m(k-2)(\frac{p}{2}-1)}\sim K^{m\de} R^{\frac{1}{2}(k-2)(\frac{p}{2}-1)}\le R^{\de} R^{\frac{n-1}{2}p-n}.
\end{equation}
We also use the following lemma.
\begin{lemma}[Estimate for narrow part at the final stage]\label{narrow}
Let $\{\si_m\}$ be $R^{-1/2}$-caps that forms a finitely overlapping cover of $\ZS^{n-1}$. Then for any $2 \leq p < \infty$,
\begin{equation}
\sum_{\si_m}\int |\Sq\vf_{\si_m}|^p\om_{B_{R}} \le C_\e R^{p\e}\int  |f|^p\om_{B_{R}}.
\end{equation}
\end{lemma}
\begin{remark}
\rm

The three dimensional version of the above lemma was proved in \cite{gan2021new} Lemma 3.9. The proof for the higher dimensional version is similar.
\end{remark}

We apply Lemma \ref{narrow} to the first term in \eqref{redc5} to get
\begin{equation}
    \int|\Sq \vf|^pw_{B_R}\lesssim_{\varepsilon} R^{\varepsilon}R^{\frac{n-1}{2}p-n}\|f\|_{L^2(w_{B_R})}^2
    \|f\|_{L^{\infty}(\R^n)}^{p-2}.
\end{equation} 
This finishes the proof of Lemma \ref{reduction}.
\end{proof}

Now everything boils down to the proof of \eqref{broadinequality}.
Actually, we will break $B_R$ into smaller pieces that adapt to the size of $\si$. We call these smaller pieces rescaled balls. More precisely, we give the definition:
\begin{definition}[Rescaled balls]
\label{recalsed-ball}
Let $\si$ be an $M^{-1}$-cap as above. For any radius $r>1$, we say a geometric object in $\ZR^{n}$ is a ``rescaled ball" of radius $r$ (or rescaled $r$-ball), if it is a tube of length $r$ and radius $M^{-1}r$, pointing to the direction $c_\si$. We use $\cp(x,r)$ to denote a rescaled $r$ ball centered at $x\in\ZR^n$. We sometimes simply write it as $\cp_r$ when we only care about its radius.
\end{definition}

We further reduce \eqref{broadinequality} to a more local version:
\begin{equation}\tag{$\mathcal{BL}_{k,A}^p$}
    \| \Sq \vf_\si \|_{\BL^p_{k,A}(\cp_R)}^p\lesssim_\e R^{p\e}R^{\frac{n-1}{2}p-n}M^{2n-(n-1)p}
    \|\Sq \vf_\si\|_{L^2(\om_{\cp_{R}})}^2
    \|f\|_{L^{\infty}(\R^n)}^{p-2}.
    \label{broadinequality2}
\end{equation}
Note that \eqref{broadinequality2} implies \eqref{broadinequality} by summing over $\cp(x,R)\subset B_R$.
In the rest of the paper, we focus on the proof of \eqref{broadinequality2}.

From now on our cap $\si$ and its scale $M$ are fixed. In the rest of the paper, we simply write $\vf=\vf_\si$,
assuming $\vf=\{f_1,\cdots,f_R\}$ satisfying $\textup{supp}(\wh f_j)\subset N_{R^{-1}}\Ga_j(\si)$.

\section{Wave packet decomposition}
\label{wave-packet}

In this section, we discuss the wave packet decomposition. Fix a scale $\rho\ge R^\e$. Let $\vg=\{g_1,\cdots,g_R\}$ be a vector-valued function with $\textup{supp}(\wh g_j)\subset N_{R^{\beta}\rho^{-1}}\Ga_j$. We build the wave packet decomposition for $\vg$ at scale $\rho$ as follows.

Recall we have chosen $\rho^{-1/2}$-caps $\{\tau\}$, defined the partition of unity $\{\wh\vp_{j,\tau}\}_\tau$ and the function $g_{j,\tau}=\vp_{j,\tau}*g$ (see \eqref{gtau} and the text there).
So, we can partition each $g_j$ of $\vg$ in the frequency space as
\begin{equation}\label{continue}
    g_j=\sum_\tau g_{j,\tau}.
\end{equation} 

Next, let us partition the physical space. Fix a cap $\tau$. After rotating $c_\tau$, without loss of generality we assume $c_\tau=e_n$. We can choose a partition of unity $\{\zeta_v\}_{v\in\Z^n}$ such that for any $v\in\Z^n$, $\zeta_v(x)=\zeta(x-v)$, where $\zeta:\ZR^n\to\ZR^+$ is a smooth function whose Fourier transform is supported in the unit ball in the frequency space, and $\zeta$ decays rapidly outside the unit ball in the physical space.

To handle the rapidly decaying tail, we modify the partition of unity $\{\zeta_v\}$ a little bit. Recall the tiny number $\be=\e^{1000}$. For each point $u\in R^\be \Z^n$, we define 
\begin{equation}
    \eta_u(x)=\sum_{v\in E(u)}\zeta_v(x),\  \textup{for~} E_u:=\{v:-R^\be/2<v_l-u_l\leq R^\be/2,~l=1,\cdots,n\}.
\end{equation}
Hence $\{\eta_u\}_{u\in R^\be \Z^n}$ also forms a smooth partition of unity. The advantage of using this new partition of unity is that the supports of functions $\{\eta_u\}$ are essentially disjoint in the following sense: if $u,u'\in R^\be \Z^n$ satisfy $|u-u'|\ge 4^n R^\be$, then $\min\limits_x (|\eta_u(x)|,|\eta_{u'}(x)|)=\rap(R)$. This means that each $\eta_u$ essentially correlates with a bounded number (at most $4^n$) of other $\eta_{u'}$.

For each $u\in R^\be\Z^n$, we consider the rescaled function 
$$\wt\eta_u(x):=\eta_u(\rho^{-1/2}x_1,\cdots,\rho^{-1/2} x_{n-1},R^\be\rho^{-1}x_n).$$ 
From the discussion above, we see $\wt{\eta}_u$ is essentially supported in a rectangle of dimensions $R^\be\rho^{1/2}\times\cdots\times R^\be\rho^{1/2}\times \rho$. We also call it a $\rho$-tube. This tube is defined by
\begin{equation}
    T_u:=\{x\in\ZR^n,|x_l-u_l|\leq R^\be\rho^{1/2},~l=1,\cdots,n-1;~|x_n-u_n|\leq \rho\}.
\end{equation}
For any $x\in\ZR^n\setminus 2T_u$, we have
\begin{equation}\label{rapiddec}
    \wt\eta_u(x)\lesssim \rap(R) \Big(1+\frac{|x_1-u_1|}{\rho^{1/2}}+\cdots+\frac{|x_{n-1}-u_{n-1}|}{\rho^{1/2}}+\frac{|x_n-u_n|}{\rho R^{-\be}}\Big)^{-(n+1)}.
\end{equation} 
We also have that the Fourier transform of $\wt\eta_u$ is supported in a $\rho^{-1/2}\times\cdots\times \rho^{-1/2}\times R^\be\rho^{-1}$-slab. This is the reason that we assume the Fourier support of the function is in $R^\be\rho^{-1}$-neighborhood of the surface.

Note that we assumed $c_\tau=e_n$. For general $\tau$, we also defined the partition of unity $\{\wt\eta_u\}$ and the tubes $\{T_u\}$, where the coreline of each tube $T_u$ is parallel to the direction $c_\tau$. We denote by 
\begin{equation}
\label{scale-R-tube-set}
    \ZT_\tau[\rho]:=\{T_u\}
\end{equation}
the collection of these tubes. In order to reveal the essential support of each $\wt\eta_u$, we use the notation 
\begin{equation}
\nonumber
    \Id_{T_u}^\ast(x):=\wt\eta_u(x).
\end{equation}

Let us continue working on the right hand side of \eqref{continue}. 
We further partition each $g_{j,\tau}$ into
$$ g_{j,\tau}=\sum_{T\in \T_\tau[\rho]}g_{j,\tau}\Id^*_{T}. $$
From now on, we just set $g_{j,T}:=g_{j,\tau}\Id_{T}^\ast$. We obtain the wave packet decomposition:

\begin{equation}
\label{wave-packet-1}
    g_j=\sum_{\tau}\sum_{T\in\ZT_\tau[\rho]}g_{j,T},\ \ \ \textup{where~}g_{j,T}:=g_{j,\tau}\Id_{T}^\ast\textup{~for~}T\in\T_\tau[\rho].
\end{equation}
If $T\in\T_\tau[\rho]$, then the Fourier transform of $\Id^*_{T}$ is supported in a slab centered at the origin which has dimensions $\rho^{-1/2}\times\cdots\times \rho^{-1/2}\times \rho^{-1}$,
and the normal direction of the this slab is $c_\tau$. So, $\wh {g_{j,T}}$ is supported in $N_{R^{\beta} \rho^{-1}}\Ga_j(2\tau)$

This is the scale $\rho$ wave packet decomposition. Each $g_{j,T}$ is a single wave packet. Since $\tau$ and $T$ are both independent to the factor $j$, we just write $\vg_\tau:=\{g_{1,\tau},\ldots,g_{R,\tau}\}$ and $\vg_{T}=\vg_{\tau}\Id^\ast_{T}:=\{g_{1,T_\tau},\ldots,g_{R,T_\tau}\}$ (when $T\in\T_\tau[\rho]$). We also call $\vg_{T}$ a single wave packet.

\smallskip
We have the following $L^2$-orthogonality for our wave packets.
\begin{lemma}[$L^2$-orthogonality]
\label{l2-orthogonality-big}
For an arbitrary collection $\ZT'\subset\cup_\tau\ZT_\tau[\rho]$, we have
\begin{equation}
    \Big\|\sum_{T\in\ZT'}g_{j,T}\Big\|_2^2\lesssim \sum_{T\in\ZT'}\|g_{j,T}\|_2^2\lesssim \|g_j\|_2^2.
\end{equation}
This estimate is uniform for all $1\leq j\leq R$.
\end{lemma}

We also have the following lemma comparing wave packets at different scales. The argument is standard, so we omit the details.

\begin{definition}\label{tuberelation}
For a tube $T$ of dimensions $\rho\times\cdots\times\rho\times r$ with $\rho\le r$, we define $\om(T)\subset \ZS^{n-1}$ to be a cap of radius $\rho/r$ whose center is the direction of $T$. We call $\om(T)$ the direction cap of $T$. For two tubes $T_i$ of dimensions $\rho_i\times\cdots\times\rho_i\times r_i$ with $\rho_i\le r_i$ $(i=1,2)$. We say $T_1<T_2$ if 
$$ T_1\subset 50 T_2\ \ \textup{and}\ \ \om(T_2)\subset 10 \om(T_1). $$
\end{definition}

\begin{lemma}\label{compare}
Fix two scales $R^\e \le \rho\le r$. Let $T_1$ be a $\rho$-tube and $T_2$ be a $r$-tube such that $T_1\not<T_2$. Then for a wave packet $g_{T_2}$, we have
$$ (g_{T_2})_{T_1}=\rap(R)\|g\|_2. $$
Similarly, for a wave packet $g_{T_1}$, we also have
$$ (g_{T_1})_{T_2}=\rap(R)\|g\|_2. $$
\end{lemma}

\section{Modified polynomial partitioning}
\label{modified-poly}

Starting in this section, we set $\de=\e^2$, $\dc=\e^6\ll\de$ and $d=d(\e)$ a sufficiently large constant depending on $\e$. The main result of this section is Lemma \ref{onestep}.

\subsection{Polynomial partitioning}

\begin{definition}
Suppose $Q_1,\ldots,Q_k$ are polynomials in $\ZR^n$. We say $Z(Q_1,\ldots,Q_k)$ is a transverse complete intersection if for any $x\in Z(Q_1,\ldots,Q_k)$, the vectors $\nabla Q_1(x),\ldots,\nabla Q_k(x)$ are linearly independent.
\end{definition}

We need a rescaled version of polynomial partitioning. The original one was proved in \cite{Guth-II} (see also in \cite{hickman2018improved}). The following one is just by rescaling.
\begin{proposition}
\label{final-partition}
Let $\rho>0$. Suppose that $g\in L^1(\R^n)$ is a non-negative function supported in $\cp_\rho\cap N_{\rho^{1/2}R^{\dc}}\bZ$, where $\cp_\rho$ is a rescaled $\rho$-ball and $\bZ$ is an $m$-dimensional transverse complete intersection of degree at most $d$. Then there exist: 
\begin{enumerate}
    \item An $(m-1)$-dimensional transverse complete intersection $\bY$ of degree $O(d)$.
    \item $\sim d^m$ many disjoint cells $O$ each of which lies in a rescaled $\rho/2$-ball. Also, any $\rho$-tube intersect at most $d+1$ cells.\label{tubeintersect} 
\end{enumerate}

\medskip

\noindent We also have:
\begin{equation}\label{partition1}
    \int_{\cp_\rho\cap N_{\rho^{1/2}R^{\dc}}\bZ} g\lesssim \sum_O \int_O g  + \int_{\cp_\rho\cap N_{\rho^{1/2}R^{\dc}}\bY}g,
\end{equation}
and 
\begin{equation}\label{partition2}
    \int_O g\lesssim d^{-m}\int_{\cp_\rho\cap N_{\rho^{1/2}R^{\dc}}\bZ} g.
\end{equation}

\end{proposition}

\subsection{Transverse equidistribution}
In this subsection, we only state the transverse equidistribution estimate. As this property already appeared in many references (for example, Section 6 of \cite{Guth-II}), we postpone the discussion to the Appendix A.

\medskip

We first give the definition of what it means for a tube $T$ to be tangent to $\bZ$.
\begin{definition}\label{tangenttube}
Let $\rho>R^\e$, $T$ be a $\rho$-tube, $\cp_\rho$ be a rescaled $\rho$-ball and $\bZ$ be a transverse complete intersection. We say $T$ is ``$\rho^{-1/2}$-tangent" to $\bZ$ in $\cp_\rho$ if the following two conditions are satisfied:
\begin{enumerate}
    \item [$\bullet$]$T\subset N_{\rho^{1/2}R^{\dc}}\bZ \cap \cp_\rho$, 
    \item [$\bullet$]If $z$ is any non-singular point of $\bZ$ lying in $10 \cp_\rho\cap 10 T$, then 
    \begin{equation}
    \label{angle-condition}
        |\angle(v(T),T_z\bZ)|\leq \rho^{-1/2}R^{\dc}.
    \end{equation}
\end{enumerate}
Here $v(T)$ is the direction of $T$, and $T_z\bZ$ is the tangent space of $\bZ$ at $z$.

\medskip

We denote the collection of these tubes by $\T_{\cp_\rho}(\bZ)$ or simply $\T_{\cp_\rho}$ when $\bZ$ is clear.
\end{definition}

\begin{remark}
\rm

Our notation is somewhat different from that in \cite{Guth-II}. In \cite{Guth-II}, Guth use different $\de_m$ for different $m$, which is also the dimension of $\bZ$, and he defined the $``\rho^{-1/2+\de_m}$-tangency" with the right hand side of \eqref{angle-condition} replaced by $\rho^{-1/2+\de_m}$. However, the proof still works if we use the same $\dc$ for all dimension $m$ and define the angle condition as in \eqref{angle-condition}.
\end{remark}

We also define what it means for a function to be concentrated on wave packets from a tube set $\T$.
\begin{definition}
Suppose that $\ZT$ is a collection of $r$-tubes. For a vector-valued function $\vg$, we say $\vg$ is concentrated on wave packets from $\ZT$ if each component $g_j$ of $\vg$ has the wave packet decomposition
\begin{equation}
    g_j=\sum_{T\in\ZT}(g_{j})_T+\rap(r)\|g_j\|_2.
\end{equation}
\end{definition}

Now we discuss our setting. Fix $M^2 \le r\le \rho$. Let $\cp_r\subset \cp_\rho$ be two rescaled balls of radius $r$ and $\rho$. Let $\bZ$ be an $m$-dimensional transverse complete intersection of degree $O(d)$. We use the notation $\T_{\cp_\rho}$($=\T_{\cp_\rho}(\bZ)$) as is Definition \ref{tangenttube}. The transverse equidistribution estimate is as follows:
\begin{proposition}[Transverse equidistribution estimate]
\label{TE}
Assume $\rho\geq r\geq \rho^{1/2}M$. Suppose $\vg=\{g_1,\cdots,g_R\}$ satisfies $\textup{supp}(\wh g_j)\subset N_{R^{\beta}\rho^{-1}}\Ga_j(\si)$ for each $j$, and $g_j$ is concentrated on wave packets from $\ZT_{\cp_\rho}$, then
\begin{equation}
\label{TE1}
    \int_{\cp_{r}\cap N_{r^{1/2}R^{\dc}}\bZ}|\Sq \vg|^2\lesssim R^{O(\dc)} \big( \frac{\rho}{r} \big)^{-\frac{n-m}{2}-1} \int_{\R^n} |\Sq \vg|^2+\rap(r)\|\Sq\vg\|_2^2.
\end{equation}
\end{proposition}
We discuss the proof of Proposition \ref{TE} in Appendix A.

\subsection{One-step polynomial partitioning}
In this subsection, we state the one-step polynomial partitioning. Compared to the three dimensional counterpart in \cite{gan2021new}, 
the difference is we break the cells obtained from Proposition 
\ref{final-partition} further into smaller cells and use the transverse equidistribution estimate to handle them. 

\begin{lemma}[One-step polynomial partitioning]
\label{onestep}Given the inputs as follows:\hfill
\rm

{\bf Inputs: } $(\rho,m,\bZ,U,\vg_U,A)$. More precisely our inputs are:
\begin{enumerate}
    \item A scale $\rho$ ($M^2 R^{\e/10n}\le \rho\le R$);
    \item A number $1\le m\le n$;
    \item An $m$-dimensional variety $\bZ$ which is a transverse complete intersection of degree $O(d)$.
    \item A set $U$ which is contained in a rescaled $\rho$-ball $\cp_U$;
    \item A vector-valued function $\vg_U=\{ g_{1,U},\cdots,g_{R,U} \}$. Each component $g_{j,U}$ has Fourier support in $N_{R^{\beta}\rho^{-1}}\big(\Gamma_j(\si)\big)$, and the wave packets of $g_{j,U}$ are $\rho^{-1/2}$-tangent to $\bZ$ in $\cp_U$;
    \item An integer $A\geq\log\log R$. 
\end{enumerate}

We obtain the outputs:

{\bf Outputs: }
We obtain an $(m-1)$-dimensional transverse complete intersection $\bY$ of degree $O(d)$. Also, we are in one of the three cases: cellular case, transverse case and tangent case.

\textit{ Cellular case}:
\begin{enumerate}
    \item We obtain a collection of cells $\co=\{O\}$ which we call cellular cells. They satisfy: each $O$ is contained in a  rescaled $\rho/2$-ball $\cp_O$ and $O=\cp_O\cap N_{(\rho/2)^{1/2}R^{\dc}}\bZ$.
    
    \item \label{cellular0} We obtain tube sets $\{\T_O\}_{O\in\co}$ where each $\T_O$ consists of some $(\rho/2)$-tubes that are $(\rho/2)^{-1/2}$-tangent to $\bZ$ in $\cp_O$. We also obtain a set of functions $\{\vg_{O}\}_{O\in\co}$ which are indexed by $\co$. Each $\vg_O$ is concentrated on wave packets from $\T_O$, and they satisfy two $L^2$-estimates
    \begin{equation}\label{cellular1}
        \sum_{O\in\co}\|\Sq\vg_{O}\|_2^2\lesssim d \|\Sq\vg_U\|^2_2,
    \end{equation}
    \begin{equation}\label{cellular2}
        \|\Sq\vg_{O}\|_2^2\lesssim d^{-(m-1)}  \|\Sq\vg_U\|^2_2,
    \end{equation}
    and a broad norm estimate:
    \begin{equation}\label{cellular3}
        \| \Sq\vg_U \|_{\BL^p_{k,A}(U)}^p\lesssim \sum_{O\in\co} \| \Sq\vg_O \|_{\BL^p_{k,A}(O)}^p. 
    \end{equation}
    \item The Fourier transform of each component $g_{j,O}$ of $\vg_{O}$ satisfies:
    \begin{equation}\label{cellular4}  
        \textup{supp~}\widehat{g_{j,O}} \subset N_{R^{\beta}(\rho/2)^{-1}} (\Ga_{j}(\si)).  
    \end{equation}
    
\end{enumerate}

\textit{Transverse case}:
\begin{enumerate}
    \item \label{transverse0}We obtain a collection of cells $\cb=\{B\}$ which we call transverse cells. They satisfies: each $B$ is a subset of $N_{(\rho R^{-\de})^{1/2}R^{\dc}}\bY\cap U$ and each $B$ is contained in a rescaled $\rho R^{-\de}$-ball $\cp_B$.
    
    \item \label{transverse00}We obtain tube sets $\{\T_{B,trans}\}_{B\in\cb}$ where each $\T_{B,trans}$ consists of some $\rho R^{-\de}$-tubes that are $(\rho R^{-\de})^{-1/2}$-tangent to $\bZ_B=\bZ+b_B$ (which is some translation of $\bZ$) in $\cp_B$. We also obtain a set of functions $\{\vg_{B,trans}\}_{B\in\cb}$ which are indexed by $\cB$. Each $\vg_{B,trans}$ is concentrated on wave packets from $\T_{B,trans}$, and they satisfy two $L^2$-estimates
    \begin{equation}\label{transverse1}
    \sum_{B\in\cb}\|\Sq\vg_{B,trans}\|^2_2\lesssim {\poly} (d) R^{-\de}\|\Sq\vg_{U}\|^2_2,
    \end{equation}
    \begin{equation}\label{transverse2}
    \|\Sq\vg_{B,trans}\|^2_2\lesssim R^{O(\dc)} R^{-\de\frac{n-m}{2}} R^{-\de}\|\Sq\vg_{U}\|^2_2,
    \end{equation}
    and a broad estimate
    \begin{equation}\label{transverse3}
    \| \Sq\vg_U \|^p_{\BL^p_{k,A}(U)}
        \lesssim \log R \sum_{B\in\cb}\| \Sq\vg_{B,trans} \|^p_{\BL^p_{k,A/2}(B)}.
    \end{equation}
    \item The Fourier transform of each component $g_{j,B,trans}$ of $\vg_{B,trans}$ satisfies:
    \begin{equation}\label{transverse4}
        \textup{supp~}\widehat{g}_{j,B,trans} \subset N_{R^{\beta}(\rho R^{-\de})^{-1}} (\Ga_{j}(\si)).  
    \end{equation}
    
\end{enumerate}

\textit{Tangent case}:
\begin{enumerate}

    \item \label{tangent0}We obtain a collection of cells $\cb=\{B\}$ which we call tangent cells. They satisfies: each $B$ is a subset of $N_{(\rho R^{-\de})^{1/2}R^{\dc}}\bY\cap U$ and each $B$ is contained in a rescaled $\rho R^{-\de}$-ball $\cp_B$.
    
    \item \label{tangent0.5}We obtain tube sets $\{\T_{B,tang}\}_{B\in\cb}$ where each $\T_{B,tang}$ consists of some $\rho R^{-\de}$-tubes that are $(\rho R^{-\de})^{-1/2}$-tangent to some translation of $\bY$ in $\cp_B$. We also obtain a set of functions $\{\vg_{B,tang}\}_{B\in\cb}$ which are indexed by $\cB$. Each $\vg_{B,tang}$ is concentrated on wave packets from $\T_{B,tang}$, and they satisfy two $L^2$-estimates
    \begin{equation}\label{tangent1}
    \sum_{B\in\cb}\|\Sq\vg_{B,tang}\|^2_2\lesssim  R^{O(\de)}\|\Sq\vg_{U}\|^2_2,
    \end{equation}
    \begin{equation}\label{tangent2}
    \|\Sq\vg_{B,tang}\|^2_2\lesssim  R^{O(\de)}\|\Sq\vg_{U}\|^2_2,
    \end{equation}
    and a broad estimate
    \begin{equation}\label{tangent3}
    \| \Sq\vg_U \|^p_{\BL^p_{k,A}(U)}
        \lesssim R^{O(\de)} \sum_{B\in\cb}\| \Sq\vg_{B,tang} \|^p_{\BL^p_{k,A/2}(B)}.
    \end{equation}
    \item The Fourier transform of each component $g_{j,B,tang}$ of $\vg_{B,tang}$ satisfies:
    \begin{equation}\label{tangent4}
        \textup{supp~}\widehat{g}_{j,B,tang} \subset N_{R^{\beta}(\rho R^{-\de})^{-1}} (\Ga_{j}(\si)).  
    \end{equation}
    
\end{enumerate}
\end{lemma}

\begin{proof}

The rest of this section is devoted to the proof of the Lemma \ref{onestep}.
We apply Proposition \ref{final-partition} to the function $$ F= \sum_{\cp_{M^2K^2}} \| \Sq \vg_U \|_{\BL^p_{k,A}(\cp_{M^2 K^2})}^p\frac{1}{|\cp_{M^2 K^2}|} \Id_{\cp_{M^2 K^2}\cap U}. $$ 
Since $\vg_U$ has wave packets concentrated on $N_{\rho^{1/2}R^{\dc}}\bZ$, we can assume $U\subset\cp_U\cap N_{\rho^{1/2}R^{\dc}}\bZ$. From Proposition \ref{final-partition} we obtain an $(m-1)$-dimensional complete intersection $\bY$ of degree $O(d)$. Also, we obtain a collection of cells $\co'=\{O'\}$ such that: $|\co'|\sim d^m$, each $O'$ is contained in a rescaled $\rho/2$-ball, and 
\begin{equation}\label{cellequal}
\| \Sq\vg_U\|^p_{\BL^p_{k,A}(O')}\lesssim d^{-m} \| \Sq\vg_U\|^p_{\BL^p_{k,A}(U)}
\end{equation}
for each cell $O'$. 

Define the wall $W:=N_{\rho^{1/2}R^{\dc}}\bY\cap U$. By \eqref{partition1}, we have the following inequality:
\begin{equation}
\label{cell-only-1}
    \| \Sq\vg_U \|^p_{\BL^p_{k,A}(U)} \lesssim \sum_{O'\in\co'}\| \Sq\vg_U \|^p_{\BL^p_{k,A}(O')}+\| \Sq\vg_U \|^p_{\BL^p_{k,A}(W)}.
\end{equation}
Invoking the wave packet decomposition at scale $\rho$, we can write 
\begin{equation}
\label{wpt-algorithm}
    \vg_U=\sum_{T\in\T} (\vg_U)_{T}+\rap(R)\|\Sq\vg\|_2,
\end{equation}
where $\T$ is a set of $\rho$-tubes that are $\rho^{-1/2}$-tangent to $\bZ$ in $\cp_U$. In the following discussion, we will first define $\T_{O'}, \T_{B',trans}$ and $\T_{B',tang}$ which are subsets of $\T$. Then we will define smaller tubes $\T_{O}, \T_{B,trans}$ and $\T_{B,tang}$ that we want.

First, we analyze the first term on the right hand side of \eqref{cell-only-1}.
For each cell $O'\in\co'$, define 
$$\ZT_{O'}:=\{ T\in\T:O'\cap T\not=\varnothing \}.$$ 
Recalling \eqref{tubeintersect} in Proposition \ref{final-partition}, we have
\begin{lemma}\label{cellulartube}
Each $T\in\T$ belongs to at most $d+1=O(d)$ many sets $\T_{O'}$.
\end{lemma}

Noting that $\ZT_{O'}$ doesn't depend on the subscript $j$ of each component $g_j$, we define the vector-valued function $\vg_{O'}$ associated to the cell $O'$ as
\begin{equation}\label{cellularfunction}
    \vg_{O'}:=\sum_{T\in\T_{O'}}(\vg_U)_{T}.
\end{equation}

\noindent Since for any $x\in O'$ one has 
$$\vg_U(x)=\vg_{O'}(x)+\rap(R)\|\Sq\vg\|_2,$$ 
as a result 
\begin{equation}\label{cellular}
    \sum_{O'\in\co'}\| \Sq\vg_U \|^p_{\BL^p_{k,A}(O')}\lesssim \sum_{O'\in\co'}\| \Sq\vg_{O'} \|^p_{\BL^p_{k,A}(O')}+\rap(R)\|\Sq\vg\|_2^p.
\end{equation}

\medskip
Next, let us analyze the second term on the right hand side of \eqref{cell-only-1}. We choose a collection of rescaled $\rho R^{-\de}$-balls $\wt\cb=\{\wt B\}$ that form a finitely overlapping covering of $U$. Define 
\begin{equation}\label{defineB}
    B':=\wt B\cap W \textup{~and~} \cb':=\{B'\}.
\end{equation}
For each $B'$, we define $\ZT_{B',tang}$ and $\ZT_{B',trans}$ which are subsets of $\T$ as follows.
\begin{definition}
\label{tangent-definition}
$\ZT_{B',tang}$ is the set of $\rho$-tubes $T\in\T$ obeying the following two conditions:
\begin{enumerate}
    \item [$\bullet$]$T\cap B'\not=\varnothing$, 
    \item [$\bullet$]If $z$ is any non-singular point of $Z$ lying in $10 \wt B\cap 10 T$, then 
    \begin{equation}
    \label{angular-condition-1}
        |\angle(v(T),T_z\bY)|\leq \rho^{-1/2}R^{\dc}.
    \end{equation}
\end{enumerate}
\end{definition}
\noindent Here $v(T)$ is the direction of $T$.
\begin{remark}
\rm

Note that the definition of tangent tubes here is different from that in Definition \ref{tangenttube}.
\end{remark}

\begin{definition}
\label{transverse-definition}
$\ZT_{B',trans}$ is the set of $\rho$-tubes $T\in\ZT$ obeying the following two conditions:
\begin{enumerate}
    \item [$\bullet$]$T\cap B'\not=\varnothing$.
    \item [$\bullet$]There exists a non-singular point $z$ of $Z$ lying in $10 \wt B\cap 10 T$, such that 
    \begin{equation}
    \label{angular-condition-2}
        |\angle(v(T),T_z\bY)|> \rho^{-1/2}R^{\dc}.
    \end{equation}
\end{enumerate}
\end{definition}
\noindent The following lemma for the transverse tubes was proved in \cite{Guth-R3}.
\begin{lemma}\label{transversetube}
Each $T\in\T$ belongs to at most ${\poly}(deg(\bY))={\poly}(d)$ many sets $\{\T_{B',trans}\}_{B'}$.
\end{lemma}

Now we define  $\vg_{B',trans}$ and $\vg_{B',tang}$ as
\begin{equation}
\label{tangent-transverse-vector}
    \vg_{B',trans}:=\sum_{T\in\T_{B',trans}}(\vg_U)_{T},\hspace{5mm}\vg_{B',tang}:=\sum_{T\in\T_{B',tang}}(\vg_U)_{T},
\end{equation}
so we have
\begin{equation}
    \vg_U(x)=\vg_{B',trans}(x)+\vg_{B',tang}(x)+\rap(R)\|\Sq\vg\|_2\ \ \ \textup{for~}x\in B'.
\end{equation}
By the triangle inequality for the broad norm \eqref{broad-triangle}, we have
\begin{align}
\label{broadness-algebraic}
    \| \Sq\vg_U \|^p_{\BL^p_{k,A}(B')}\lesssim\, &\| \Sq\vg_{B',trans} \|^p_{\BL^p_{k,A/2}(B')}+\| \Sq\vg_{B',tang} \|^p_{\BL^p_{k,A/2}(B')}\\ \nonumber
    &+\rap(R)\|\Sq\vg\|_2^p.
\end{align}
Consequently, one has
\begin{align}
\label{transverse}
    \| \Sq\vg_U \|^p_{\BL^p_{k,A}(W)}\!\lesssim & \sum_{B'}\| \Sq\vg_{B'\!,trans} \|^p_{\BL^p_{k,A/2}(B')}\!\!+\!\sum_{B'}\| \Sq\vg_{B'\!,tang} \|^p_{\BL^p_{k,A/2}(B')}\\ \nonumber
    &+\rap(R)\|\Sq\vg\|_2^p.
\end{align}

Combining \eqref{cell-only-1}, \eqref{cellular} and \eqref{transverse}, we get
\begin{align}
\label{algorithm-crude}
     \| \Sq\vg_U \|^p_{\BL^p_{k,A}(U)}\lesssim&\sum_{O'\in\co'} \| \Sq\vg_{O'} \|^p_{\BL^p_{k,A}(O')}+\sum_{B'\in\cb'}  \| \Sq\vg_{B',trans} \|^p_{\BL^p_{k,A/2}(B')}\\ \nonumber
    &+\sum_{B'\in\cb'}  \| \Sq\vg_{B',tang} \|^p_{\BL^p_{k,A/2}(B')} +\rap(R)\|\Sq\vg\|_2^p.
\end{align}
We can just ignore the rapidly decaying term.
Now we determine which one of the three cases we are in according to which term on the right hand side of \eqref{algorithm-crude} dominates.

\smallskip

\textit{Transverse case}: If the second term on the right hand side of \eqref{algorithm-crude} dominates, we say ``we are in the transverse case".

In this case, we cannot directly use the cells $B'$ for $B$, since $B'$ is too thick to satisfies condition \eqref{transverse00}. We need to break each $B'$ into smaller cells $\{B\}$
and define for each $B$ a set of $\rho R^{-\de}$-tubes $\T_{B,trans}$, so that $\T_{B,trans}$ is $(\rho R^{-\de})^{-1/2}$-tangent to some translation of $\bZ$ in $B$. The two main tools are the probability method and the transverse equidistribution estimate. We incorporate them in the next lemma.

\begin{lemma}\label{tetrans}
For a fixed $B'$, 
we can find a set $\fB\subset B(0,\rho^{1/2})$ and disjoint sets of $\rho R^{-\de}$-tubes $\{\T_b\}_{b\in\fB}$,
such that the tubes in $\T_b$ are $(\rho R^{-\de}) ^{-1/2}$-tangent to $b+\bZ$ in $B'$ for each $b\in\fB$.
Intuitively, one may think $$N_{\rho^{1/2}R^{\dc}}\bZ \cap B'\approx\bigcup_{b\in\fB}N_{(\rho R^{-\de})^{1/2}R^{\dc}}(b+\bZ)\cap B'.$$ If we define
\begin{equation}
    \cb(B'):=\{ B'\cap N_{(\rho R^{-\de})^{1/2}}(b+\bZ) \}_{b\in\fB}=\{B\},
\end{equation}
and for each $B=B'\cap N_{(\rho R^{-\de})^{1/2}R^\dc}(b+\bZ)\in\cb(B')$, define
\begin{align}
 \label{transvec}
    \T_{B,trans}:=\T_b,\hspace{.5cm}\vg_{B,trans}:=\sum_{T'\in\T_{B,trans}} (\vg_{B',trans})_{T'},
\end{align}
then we have the following estimates:
\begin{align}
    \label{tetrans1}\| \Sq\vg_{B',trans} \|^p_{\BL^p_{k,A/2}(B')}&\lesssim \log R \sum_{B\in\cb(B')} \| \Sq\vg_{B,trans} \|^p_{\BL^p_{k,A/2}(B)},\\
    \label{tetrans2}\sum_{B\in\cb(B')} \| \Sq\vg_{B,trans} \|^2_2&\lesssim R^{-\de} \| \Sq\vg_{B',trans} \|^2_2,\\ \label{tetrans3}
    \| \Sq \vg_{B,trans} \|_2^2 &\lesssim R^{O(\dc)} R^{-\de\frac{n-m}{2}}R^{-\de} \|\Sq \vg_{B',trans}\|_2^2.
\end{align}

\end{lemma}

\begin{remark}
{\rm
The proof of \eqref{tetrans1} is by some probability argument, where we randomly choose the set $\fB\in B(0,\rho^{1/2})$. The argument can be found in \cite{Guth-II} \textup{page 132}, or \cite{guth2019sharp} \textup{Lemma 10.5}. The factor $\log R$ comes from the pigeonhole argument. For safety, we may choose the factor to be $(\log R)^{10}$, but it is acceptable as long as it is much smaller than $R^{\dc}$. So we just put $\log R$ here for simplicity.

The proof of \eqref{tetrans2} is by a standard $L^2$-argument. We gain a factor $R^{-\de}$ because $\vg_{B',trans}$ is a sum of $\rho$-wave packets whereas $\vg_{B,trans}$ is a sum of $\rho R^{-\de}$-wave packets. 

We also give the idea of the proof of \eqref{tetrans3}. For a fixed $B=B'\cap N_{(\rho R^{-\de})^{1/2}R^\dc}(b+\bZ)\in\cb(B')$, we apply Proposition \ref{TE} with $\vg=\vg_{B',trans}$, $(\rho,r)=(\rho,\rho R^{-\de})$, $\cp_\rho=\cp_U$, $\cp_r=\cp_{B'}$ and ignore the rapidly decaying term to obtain
$$ \int_{B'\cap N_{(\rho R^{-\de})^{1/2}R^{\dc}}(b+\bZ)}|\Sq \vg_{B',trans}|^2\lesssim R^{O(\dc)} \big( R^\de \big)^{-\frac{n-m}{2}-1} \int|\Sq \vg_{B',trans}|^2. $$
Another key observation is $\Sq \vg_{B,trans}(x)\sim \Id_{B' \cap N_{(\rho R^{-\de})^{1/2}R^\dc}(b+\bZ)} \cdot \Sq \vg_{B',trans}(x)$, which is the same as (8.20) in \cite{Guth-II}. Plugging into the left hand side of the above inequality, we obtain \eqref{tetrans3}.

}
\end{remark}

Now we define $\cB:=\cup_{B'}\cB(B')$. Since we are in the transverse case, in order to prove \eqref{transverse3}, one just needs to sum up all $B\in\cB$ in \eqref{algorithm-crude} and \eqref{tetrans1}. Also, we see that \eqref{tetrans3} verifies \eqref{transverse2}, and \eqref{transverse4} is a fact from the wave packet decomposition.

It remains to show \eqref{transverse1}. From \eqref{tangent-transverse-vector} and Lemma \ref{transversetube} one has
\begin{equation}
\nonumber
\begin{split}
    \sum_{B'\in\cb'}\|\Sq\vg_{B',trans}\|_2^2=\sum_{B'\in\cb'}\big\|\Sq\big(\sum_{T\in\T_{B',trans}}(\vg_U)_{T}\big)\big\|_2^2
    &\lesssim {\poly}(d)\big\|\Sq\big(\sum_{T\in\T}(\vg_U)_{T}\big)\big\|^2_2\\
    &\lesssim{\poly}(d)\|\Sq\vg_U\|^2_2.
\end{split}
\end{equation}
Combine this with \eqref{tetrans2}, we prove
\eqref{transverse1}.

\smallskip

\textit{Tangent case}: If the third term on the right hand side of \eqref{algorithm-crude} dominates, we say ``we are in the tangent case". 

The tangent case is handled in a similar way as in the transverse case.
Recalling the definition of $\T_{B',tang}$ in Definition \ref{tangent-definition}, the tubes in $\T_{B',tang}$ are in some sense tangent to $\bY$ in $B'$. We can derive a similar lemma as Lemma $\ref{tetrans}$ with $\bZ$ replaced by $\bY$. Then, all the argument work in the same way.

We also remark that we only need to care about the requirement \eqref{tangent0.5}, since the other estimates \eqref{tangent1}, \eqref{tangent2} and \eqref{tangent3} are quite crude and can be easily verified. The reason that we only need crude estimates in tangent case is because we will only encounter at most $n-k$ tangent cases in our iteration process.

\smallskip
\textit{Cellular case}:
If the first term on the right hand side of \eqref{algorithm-crude} dominates, we say ``we are in the cellular case". We have
\begin{equation}
\nonumber
    \| \Sq\vg_U \|^p_{\BL^p_{k,A}(U)}  \lesssim\sum_{O'\in\co'}\| \Sq\vg_{O'} \|^p_{\BL^p_{k,A}(O')}.
\end{equation}
Since $\T_{O'}$ is a collection of $\rho$-tubes, we need to define $\rho/2$-tubes $\T_{O}$ in order to satisfy the requirement in \eqref{cellular0}. 
However, this is much easier than the transverse case, since the adjacent scales are $\rho$ and $\rho/2$ instead of $\rho$ and $\rho R^{-\de}$. In fact, at a cost of some absolute constant that comes from using the triangle inequality, we can assume that $\T_{O'}$ is a collection of $\rho/2$-tubes, each of which is $(\rho/2)^{-1/2}$-tangent to $\bZ$, and \eqref{cellularfunction} is a wave packet decomposition at scale $\rho/2$, without loss of generality.

Recall that  $|\co'|\sim d^m$ and \eqref{cellequal}. Via pigeonholing, we can choose a subset of $\co'$, which is denoted by $\co$, such that 
\begin{equation}
    \|\Sq\vg_U\|^p_{\BL^p_{k,A}(U)}\sim d^m \|\Sq\vg_O\|^p_{\BL^p_{k,A}(O)}
\end{equation}
holds for every $O\in \co$, and $|\co|\sim d^m$. This verifies \eqref{cellular3}. 

To show \eqref{cellular1}, we first recall \eqref{cellularfunction}, so
\begin{equation}
\nonumber
    \sum_{O\in\co}\|\Sq\vg_{O}\|_2^2=\sum_{O\in\co}\big\|\Sq\big(\sum_{T\in\T_{O}}(\vg_U)_{T}\big)\big\|^2.
\end{equation}
By Lemma \ref{cellulartube}, we see that each tube $T$ belongs to $O(d)$ many sets $\T_{O'}$, implying
\begin{equation}
    \sum_{O\in\co}\big\|\Sq\big(\sum_{T\in\T_{O}}(\vg_U)_{T}\big)\big\|_2^2\lesssim d\big\|\Sq(\sum_{T\in\T}(\vg_U)_{T})\big\|^2_2 \lesssim d\|\Sq\vg_U\|^2_2.
\end{equation}
Combining the above two inequalities, we proved \eqref{cellular1}.

To prove \eqref{cellular2}, we need to refine $\co$ once more. Note that we just proved
$$  \sum_{O\in\co}\|\Sq\vg_{O}\|_2^2\lesssim d\|\Sq\vg_U\|^2_2,$$
Since $|\co|\sim d^m$, we see the number of cells $O\in\co$ for which $\|\Sq\vg_{O}\|_2^2\lesssim d^{-(m-1)}\|\Sq\vg_U\|^2_2$ is $\sim d^m$.
We still denote these cells by $\co$, then \eqref{cellular2} holds. Finally, \eqref{cellular4} is a fact from the wave packet decomposition.
\end{proof}

\section{The first algorithm}

In this section, we discuss our first algorithm.
This algorithm is processed by iteratively using Lemma \ref{onestep}. At each step of the iteration, we endow one of the states: cellular state, transverse state and tangent state.
The iteration ends when we arrive in the tangent state, or the scale is very small (slightly larger than $M^2$). We will discuss more carefully about the latter two scenarios.

\begin{algorithm}[The first algorithm]\label{1algorithm}\hfill
{\rm

{\bf Inputs: } $(r,m,p_m,A)$ and $\{(\bZ_{O_0},O_0,\vf_{O_0})\}_{O_0\in\cO_0}$. More precisely, the quadruple consists of:
\begin{enumerate}
    \item A scale $r$ ($M^2R^{\e/10n}\le r\le R$);
    \item A number $1\le m\le n$;
    \item A number $p_m$ which is used for $L^{p_m}$ space (we will just write $p$ instead of $p_m$ throughout this algorithm);
    \item An integer $A\geq \log\log R$. 
\end{enumerate}

\noindent
For each $O_0\in\co_0$, $(\bZ_{O_0},O_0,\vf_{O_0})$ consists of:
\begin{enumerate}
    \item An $m$-dimensional variety $\bZ_{O_0}$ which is a transverse complete intersection of degree $O(d)$;
    \item A cell $O_0=\cp_{O_0}\cap N_{r^{1/2}}\bZ_{O_0}$, where $\cp_{O_0}$ is a rescaled $r$-ball;
    \item A vector-valued function $\vf_{O_0}=\{ f_{1,O_0},\cdots,f_{R,O_0} \}$. Each component $f_{j,O_0}$ has Fourier support in $N_{R^{\beta}r^{-1}}\big(\Gamma_j(\si)\big)$, and the wave packets of $f_{j,O_0}$ are $r^{-1/2}$-tangent to $\bZ_{O_0}$ in $\cp_r$.
\end{enumerate}

Then we have the following {\bf outputs}:

\medskip

\noindent$\bullet$ There exists an integer $s\ge 1$ which denotes the total number of iteration steps. There is a function STATE which we use to record the state of each step:
\begin{equation}
    \textup{STATE}:\{1,2,\cdots,s\}\rightarrow \{\textup{cell,~trans,~tang}\}.
\end{equation}
We require the tangent case appear at most once, and if it appears, it should only appear at the last step. That is: $\textup{STATE}(u)=$ tang implies $u=s$. If the tangent case does not appear, then we end with a small radius $r_s\sim M^2R^{\e/10n}$ ($r_s$ is as below).

\medskip

\noindent$\bullet$ At each step $u$, $u\in\{1,\cdots,s\}$, we have:

\smallskip

\noindent{\bf 1.} A scale $r_u$ whose explicit formula is
\begin{equation}\label{it1.1}
    r_u=r 2^{-c(u)}R^{-\de a(u) },
\end{equation}
where the two parameters $c(u)$ and $a(u)$ are defined by
\begin{align}
    \label{it1.11}&c(u):=\#\{1\le i\le u:~\textup{STATE}(i)=\textup{cell}\}, \\ 
    \label{it1.12}&a(u):=\#\{1\le i\le u:~\textup{STATE}(i)=\textup{trans}\}.
\end{align}

\smallskip

\noindent{\bf 2.} A number $A_u$ defined by
\begin{equation}\label{it1.2}
    A_u=A/2^{a(u)}.
\end{equation}
For convenience we also set $A_0=A$. The number $A_u$ is used for the broad norm $\BL^p_{k,A_u}$ at step $u$.

\smallskip

\noindent{\bf 3.} A set of cells $\co_u=\{O_u\}$ such that each $O_u$ is contained in a rescaled $r_u$-ball $\cp_{O_u}$. Each $O_u$ has a unique parent $O_{u-1}\in\co_{u-1}$. We denote this relation by 
\begin{equation}
    O_u<O_{u-1}.
\end{equation}
Moreover we have the nested property for these cells. That is, for any cell $O_{s}\in\co_{s}$, there exist unique $O_u\in \co_u$ $(u=0,1,\cdots,s)$ such that
\begin{equation}
    O_s<O_{s-1}<\cdots<O_1<O_0. 
\end{equation}
For $O_u$ and $O_{u'}$ in this chain with $u>u'$,
we also write $O_u<O_{u'}$.

\noindent{\bf 4.} A set of $r_u$-tubes $\T_{O_u}$ and a set of functions $\{\vf_{O_u}\}_{O_u\in\co_{u}}$ that satisfy the following iteration formula
\begin{equation}\label{it3}
    \vf_{O_u}:=\sum_{T\in\T_{O_u}}(\vf_{O_{u-1}})_T+\rap(r_u)\|\Sq\vf\|_2.
\end{equation}
Here tubes in $\T_{O_u}$ are $r_u^{-1/2}$-tangent to a translated variety $\bZ_{O_u}(=b_{O_u}\!+\!\bZ)$ in $\cp_{O_u}$.

\smallskip

\noindent{\bf 5.} There are three possible cases for each step $u$: cellular case, transverse case and tangent case. The outputs for each case are the following:

\medskip

\noindent\textit {Cellular state}: If $\textup{STATE}(u)=$ cell, we have the following outputs.

\begin{enumerate}
    \item[i.] We have the following $L^2$-relations between two adjacent steps:
    \begin{align}\label{itcellular1}
        \sum_{O_u}\|\Sq\vf_{O_{u}}\|^2_2\lesssim & d \sum_{O_{u-1}}\|\Sq\vf_{O_{u-1}}\|^2_2,\\ \label{itcellular2}
        \|\Sq\vf_{O_{u}}\|^2_2\lesssim d^{-(m-1)}& \|\Sq\vf_{O_{u-1}}\|^2_2\ \ \ \textup{for~} O_u<O_{u-1}.
    \end{align}

    \item[ii.] We have the broad norm estimate: 
    \begin{equation}\label{itcellular3}
        \sum_{O_{u-1}\in\co_{u-1}}\|\Sq(\vf_{O_{u-1}})\|^p_{\BL^p_{k,A_{u-1}}(O_{u-1})}\lesssim \sum_{O_{u}\in\co_u}\|\Sq(\vf_{O_u})\|^p_{\BL^p_{k,A_u}(O_u)}
    \end{equation}
    
    \item[iii.] The Fourier transform of each component $f_{j,O_u}$ of $\vf_{O_u}$ satisfies:
    \begin{equation}\label{itcellular4}
       \textup{supp} \wh{f}_{j,O_u}\subset N_{R^{\beta}r_u^{-1}}(\Ga_{j}(\si)).  
    \end{equation}
\end{enumerate}

\medskip

\noindent\textit{Transverse state}: If $\textup{STATE}(u)=$ trans, we have the following outputs.
\begin{enumerate}
    \item[i.] We have the following $L^2$-relations between adjacent steps:
    \begin{align}\label{ittransverse1}
        \sum_{O_u}\|\Sq\vf_{O_{u}}\|^2_2&\lesssim  {\poly}(d)R^{-\de}\sum_{O_{u-1}}\|\Sq\vf_{O_{u-1}}\|^2_2\\ \label{ittransverse2}
        \|\Sq\vf_{O_{u}}\|^2_2&\lesssim R^{O(\dc)} R^{-\de\frac{n-m}{2}} R^{-\de} \|\Sq\vf_{O_{u-1}}\|^2_2\ \ \ \textup{for~} O_u<O_{u-1}.    
    \end{align}
    
    \item[ii.] We have the broad norm estimate:
    \begin{equation}\label{ittransverse3}
        \sum_{O_{u-1}}\|\Sq(\vf_{O_{u-1}})\|^p_{\BL^p_{k,A_{u-1}}(O_{u-1})}\lesssim (\log R)^3 \sum_{O_{u}}\|\Sq(\vf_{O_u})\|^p_{\BL^p_{k,A_u}(O_u)}.
    \end{equation}
    
    \item[iii.] The Fourier transform of each component $f_{j,O_u}$ of $\vf_{O_u}$ satisfies:
    \begin{equation}\label{ittransverse4}
       \textup{supp} \wh{f}_{j,O_u}\subset N_{R^{\beta}r_u^{-1}}(\Ga_{j}(\si)).  
    \end{equation}
\end{enumerate}

\medskip

\noindent\textit{Tangent state}: If $\textup{STATE}(u)=$ tang, so $u=s$, then we have the following outputs.
\begin{enumerate}
    \item[i.] We have the following $L^2$-relations:
    \begin{align}\label{ittangent1}
        \sum_{O_{s}}\|\Sq\vf_{O_{s}}\|^2_2\lesssim  R^{O(\de)}\sum_{O_{s-1}}\|\Sq\vf_{O_{s-1}}\|^2_2\\ \label{ittangent2}
        \|\Sq\vf_{O_{s}}\|^2_2\lesssim R^{O(\de)}\|\Sq\vf_{O_{s-1}}\|^2_2\ \ \ \textup{for~}O_s<O_{s-1}.    
    \end{align}
    
    \item[ii.] We have broad norm estimate:
    \begin{equation}\label{ittangent3}
        \sum_{O_{s-1}}\|\Sq(\vf_{O_{s-1}})\|^p_{\BL^p_{k,A_{s-1}}(O_{s-1})}\lesssim (\log R)^3\sum_{O_s} \|\Sq(\vf_{O_s})\|^p_{\BL^p_{k,A_s}(O_s)}.
    \end{equation}
    
    \item[iii.] The Fourier transform of each component $f_{j,O_s}$ of $\vf_{O_s}$ satisfies:
    \begin{equation}\label{ittangent4}
       \textup{supp} \wh{f}_{j,O_s}\subset N_{R^{\beta}r_s^{-1}}(\Ga_{j}(\si)).  
    \end{equation}
    
    \item[iv.]\label{ittangent5} For each $O_s\in\co_s$, the tubes in $\T_{O_s}$ are $r_s^{-1/2}$-tangent to some $(m-1)$-dimensional transverse complete intersection $\bY_{O_s}$ in $\cp_{O_s}$.
    
\end{enumerate}

}
\end{algorithm}

\begin{proof} We are going to iteratively apply Lemma \ref{onestep}. 
Suppose the iteration is done for step $u$, so we obtain a scale $r_{u}$, a set of cells $\co_{u}$, tube sets $\{\T_{O_{u}}\}_{O_{u}\in\co_{u}}$ and a set of functions $\{\vf_{O_{u}}\}_{O_{u}\in\co_{u}}$. For each $O_u\in\co_u$, we apply Lemma \ref{onestep} to the tuple $(r_u,m,\bZ_{O_u},O_u,\vf_{O_u},A_u)$. To verify this is a valid input, we note that $O_u$ is contained in a rescaled $r_u$-ball and also \eqref{itcellular4} and \eqref{ittransverse4} verify the requirement on the Fourier support of $\vf_{O_u}$.  

For each $O_u$, we obtain cells $\co_{u+1}(O_u)$ and the associated tubes and functions 
\begin{equation}
\nonumber
    \{\T_{O_{u+1}}\}_{O_{u+1}\in\co_{u+1}(O_u)},\hspace{.5cm}\{\vf_{O_{u+1}}\}_{O_{u+1}\in\co_{u+1}(O_u)}.
\end{equation}
We put $O_u$ into one of $\co_{u,cell}$, $\co_{u,trans}$ and $\co_{u,tang}$, depending on which case we are in Lemma \ref{onestep}. Note that
\begin{equation}
\nonumber
    \co_u=\co_{u,cell}\sqcup\co_{u,trans}\sqcup \co_{u,tang}.
\end{equation}
From \eqref{cellular3}, \eqref{transverse3} and \eqref{tangent3}, we have
\begin{align}
    \label{3term}\sum_{O_u\in\co_u} \|\Sq\vf_{O_u}\|^{p}_{\BL^{p}_{k,A_u}}\!&\lesssim \sum_{O_u\in\co_{u,cell}}\ \sum_{O_{u+1}\in\co_{u+1}(O_u)} \|\Sq\vf_{O_{u+1}}\|^{p}_{\BL^{p}_{k,A_u}}\\
    \nonumber&+\log R \sum_{O_u\in\co_{u,trans}}\ \sum_{O_{u+1}\in\co_{u+1}(O_u)}\!\! \|\Sq\vf_{O_{u+1}}\|^{p}_{\BL^{p}_{k,A_u/2}}\\
    \nonumber&+R^{O(\de)}\sum_{O_u\in\co_{u,tang}}\ \sum_{O_{u+1}\in\co_{u+1}(O_u)}\!\! \|\Sq\vf_{O_{u+1}}\|^{p}_{\BL^{p}_{k,A_u/2}}
\end{align}

There are three possible states: cellular state, transverse state and tangent state, depending on which term on the right hand side of \eqref{3term} dominates. We discuss them separately.

\smallskip

\textit{Cellular state}: We say ``the step $u+1$ is in the cellular state", if the first term on the right hand side of \eqref{3term} dominates, i.e.
\begin{equation}\label{celldom}
    \sum_{O_u\in\co_u} \|\Sq\vf_{O_u}\|^{p}_{\BL^{p}_{k,A_u}}\lesssim \sum_{O_u\in\co_{u,cell}}\ \sum_{O_{u+1}\in\co_{u+1}(O_u)} \|\Sq\vf_{O_{u+1}}\|^{p}_{\BL^{p}_{k,A_u}}.
\end{equation}
We set
\begin{equation}
    \textup{STATE}(u+1)=\textup{cell},\ \ \ r_{u+1}=r_u /2.
\end{equation}
Also note from \eqref{it1.12}, 
\begin{equation}
    a(u+1)=a(u).
\end{equation}

Now we define
\begin{align}
\label{celdef1}
     \co_{u+1}:=&\bigcup_{O_u\in\co_{u,cell}}\co_{u+1}(O_u),\\ \label{celdef2}
     \{\T_{O_{u+1}}\}_{O_{u+1}\in\co_{u+1}}:=&\bigcup_{O_u\in\co_{u,cell}}\{\T_{O_{u+1}}\}_{O_{u+1}\in\co_{u+1}(O_u)},\\ \label{celdef3}
    \{\vf_{O_{u+1}}\}_{O_{u+1}\in\co_{u+1}}:=&\bigcup_{O_u\in\co_{u,cell}}\{\vf_{O_{u+1}}\}_{O_{u+1}\in\co_{u+1}(O_u)}.
\end{align}
We can check the following results. 
\begin{enumerate}
    \item By inequality \eqref{cellular1} and the definition in \eqref{celdef3}, we have
    \begin{align}
        \sum_{O_{u+1}\in\co_{u+1}}\|\Sq\vf_{O_{u+1}}\|^2_2&=\sum_{O_u\in\co_{u,cell}}\ \sum_{O_{u+1}\in\co_{u+1}(O_u)}\|\Sq\vf_{O_{u+1}}\|^2_2\\
        \nonumber&\lesssim d \sum_{O_{u}\in\co_{u,cell}}\|\Sq\vf_{O_{u}}\|^2_2\le d\sum_{O_{u}\in\co_u}\|\Sq\vf_{O_{u}}\|^2_2,
    \end{align}
which verifies \eqref{itcellular1}.

    \item By \eqref{cellular2}, we can verify \eqref{itcellular2}.
    
    \item By \eqref{celldom}, we can verify \eqref{itcellular3}.
    \item By \eqref{cellular4}, we can verify \eqref{itcellular4}.
\end{enumerate}
At this point we finish the proof for \textit{Cellular state}.    

\smallskip

\textit{Transverse state}: We say ``the step $u+1$ is in the transverse state", if the second term on the right hand side of \eqref{3term} dominates, i.e. 
\begin{equation}\label{transdom}
    \sum_{O_u\in\co_u} \|\Sq\vf_{O_u}\|^{p}_{\BL^{p}_{k,A_u}}\lesssim 
    \log R \sum_{O_u\in\co_{u,trans}}\ \sum_{O_{u+1}\in\co_{u+1}(O_u)} \|\Sq\vf_{O_{u+1}}\|^{p}_{\BL^{p}_{k,A_u/2}}
\end{equation}
We set
\begin{equation}
    \textup{STATE}(u+1)=\textup{trans},\ \ \ r_{u+1}=r_u R^{-\de},
\end{equation}
and note that from \eqref{it1.12},
\begin{equation}
    a(u+1)=a(u)+1.
\end{equation}

Now we define $\co_{u+1}$ and associated tubes and functions in the same way as in \eqref{celdef1} --- \eqref{celdef3}, except we replace $\co_{u,cell}$ by $\co_{u,trans}$.
We can check the following: 
\begin{enumerate}
    \item By inequality \eqref{transverse1}, we have
    \begin{align}
        \sum_{O_{u+1}\in\co_{u+1}} &\|\Sq\vf_{O_{u+1}}\|^2_2=\sum_{O_u\in\co_{u,trans}}\ \sum_{O_{u+1}\in\co_{u+1}(O_u)}\|\Sq\vf_{O_{u+1}}\|^2_2\\
        \nonumber&\lesssim \poly(d)R^{-\de} \sum_{O_{u}\in\co_{u,trans}}\|\Sq\vf_{O_{u}}\|^2_2\le d\sum_{O_{u}\in\co_u}\|\Sq\vf_{O_{u}}\|^2_2,
    \end{align}
which verifies \eqref{ittransverse1}.

    \item By \eqref{transverse2}, we can verify \eqref{ittransverse2}.
    
    \item By \eqref{transdom}, we can verify \eqref{ittransverse3}.
    \item By \eqref{transverse4}, we can verify \eqref{ittransverse4}.
\end{enumerate}
At this point we finish the proof for \textit{Transverse state}.

\smallskip  

\textit{Tangent state}: We say ``the step $u+1$ is in the tangent state", if the third term on the right hand side of \eqref{3term} dominates. Actually, one may not encounter the tangent state throughout the iteration, but once the tangent state appears, the iteration stops and so we have $u=s$. 

We can proceed in exactly the same way as we did for the \textup{transverse state}. The results we would like to verify have their counterparts in Lemma \ref{onestep}, so details are omitted here.

\smallskip
Now we have finished the proof of Algorithm \ref{1algorithm}.
\end{proof}

\section{The second algorithm}

Before we discuss the second algorithm, let us first determine the Lebesgue exponents $\{p_l\}_{l=k}^n$ that we will use later.

\subsection{Lebesgue exponents \texorpdfstring{$\{p_l\}_{l=k}^n$}{Lg}}
We will choose $\{p_l\}$ in the same way as in \cite{hickman2020note}. First, suppose that they satisfy
\begin{equation}
    2\le p=p_n\le\cdots\le p_{k+1}\le p_k, \hspace{.5cm}
    p_k=\frac{2k}{k-1}.
\end{equation}
Their precise definitions  will be given inductively by the following formula:
\begin{equation}\label{prelation}
    \frac{1}{p_l}:=\frac{1-\al_{l-1}}{2}+\frac{\al_{l-1}}{p_{l-1}}, \hspace{.5cm} l=n,n-1,\cdots,k+1,
\end{equation}
where $\alpha_k,\cdots,\alpha_{n-1}\in[0,1]$ are to be determined.

It is convenient to define $\{\be_l\}_{l=k}^n$, which are partial product of $\alpha_l$:
\begin{equation}\label{defbe}
    \beta_l:=\prod_{i=l}^{n-1}\al_i,\ \ \be_n:=1.
\end{equation}
From \eqref{prelation}, we see
\begin{equation}\label{prelation2}
    \frac{1}{2}-\frac{1}{p_n}=\be_l\big( \frac{1}{2}-\frac{1}{p_l} \big),\ \ k\le l \le n.
\end{equation}

We still need another set of numbers $\{\ga_l\}_{l=k}^n$ satisfying 
\begin{equation}
\nonumber
    \ga_l\in[0,1],\hspace{.5cm} \sum_{k\leq l\leq n}\ga_l=1.
\end{equation}
Similarly, we define $\{\si_l\}_{l=k}^{n}$ to be the partial sum of $\ga_l$:
\begin{equation}\label{defsi}
    \si_l:=\sum_{k\leq i\leq l} \ga_i.
\end{equation}
It is convenient to also set $\si_{k-1}:=0$.

What we want is the solution of following system of equations:
\begin{align}
    \label{solve1} &\frac{\be_{l+1}-\be_l}{2}-\frac{1+\si_l}{2}\big(\frac{1}{2}-\frac{1}{p_n}\big)=0\ \ (k\le l\le n-1),\\  
    \label{solve2} &\frac{\be_{l+1}}{2}-(1+l(1-\si_l))\big(\frac{1}{2}-\frac{1}{p_n}\big)=0\ \ (k-1\le l\le n-1). 
\end{align}
The system was solved in \cite{hickman2020note}. In fact, we get
\begin{align}
    \label{defga}&\ga_l=\frac{k-1}{2}\cdot \frac{1}{l(l-1)}\cdot\prod_{i=k}^l\frac{2i}{2i+1}\ \ (k\le l\le n-1),\\
    &\ga_n=1-\sum_{l=k}^{n-1}\ga_l.
\end{align}
We also get $p_n=p_n(k)$, where $p_n(k)$ is given by \eqref{defpk}.
\begin{remark}
\rm

Our subscript is slightly different from that of \cite{hickman2020note}. Our subscript $l$ indicates the dimension while their subscript indicates the codimension. 
\end{remark}

\subsection{The second algorithm}\hfill

We discuss our second algorithm here. We will use the first algorithm constantly.

\begin{algorithm}[The second algorithm] \hfill

\rm{
We begin with the cell $\cp_R$, the function $\vf$ and a number $A_n\sim\log\log R$. Recall that each component $f_j$ of $\vf$ has Fourier support in $N_{R^{\beta}R^{-1}}\Ga_j(\si)$. For convenience, we write $S_n=\cp_R$, $\vf_{S_n}=\vf$. 

\smallskip

We have the following {\bf outputs}:

\medskip
\noindent$\bullet$ There is an integer $m$ ($k\le m\le n$) so that the algorithm ends at dimension $m$.

\noindent$\bullet$
We obtain a sequence of cell sets $\cs_n,\cs_{n-1},\cdots,\cs_m,\co$ (for convenience we may write $\cs_{m-1}=\cO$).
The cells are nested in the sense that for any $O\in \co$, there exist unique $S_l\in\cs_l$ ($m\le l\le n$) such that 
$$ O<S_m<\cdots<S_{n-1}<S_n. $$
For $S_l, S_{l'}$ in this chain with $l<l'$,
we also write $S_l<S_{l'}$ to mean $S_{l'}$ is the ancestor of $S_l$.

\medskip
\noindent$\bullet$
We obtain numbers $a_l, c_l$ ($m-1\le l\le n$) where $a_l$ (resp. $c_l$) is the number of algebraic (resp. cellular) cases from $\cs_{l+1}$ to $\cs_l$. We also obtain a sequence of scales 
\begin{equation}\label{obtainradius}
    R=r_n\ge r_{n-1}\ge \cdots r_m\ge r_{m-1}\sim M^2 R^{\e/10n},
\end{equation}
so that the cells in $\cs_l$ are at scale $r_l$ ($m-1\le l\le n$). 
These scales are defined recursively as
$$r_l=r_{l+1}2^{-c_l}R^{-\de a_l}.$$

\medskip
\noindent$\bullet$
We obtain associated functions
\begin{equation}
\nonumber
    \{\vf_{S_n}\}_{S_n\in\cs_n}, \{\vf_{S_{n-1}}\}_{S_{n-1}\in\cs_{n-1}}, \cdots, \{\vf_{S_m}\}_{S_m\in\cs_m}, \{\vf_{O}\}_{O\in\co}.
\end{equation}
Also, for each $S_l\in\cs_l$, there is an $l$-dimensional transverse complete intersection $\bZ_{S_l}$. The functions satisfy: each component of $\vf_{S_l}$, $f_{j,S_l}$, has Fourier transform supported in $N_{R^{\beta}r_l^{-1}}(\Ga_j(\si))$. And $\vf_{S_l}$ is concentrated on scale $r_l$ wave packets that are $r_l^{-1/2}$-tangent to $\bZ_{S_l}$ in $\cp_{S_l}$.

\medskip
\noindent$\bullet$
We set
\begin{align}
\nonumber
    D_l:=d^{c_l}, \hspace{.5cm}
    A_l:=A_{l+1}/2^{a_l+1}.
\end{align} 
Then for $m-1\le l\le n-1$, we have the following estimates:

\begin{align}\label{est0.1}
    \sum_{S_{l+1}}\|\Sq\vf_{S_{l+1}}\|^{p_{l+1}}_{\BL^{p_{l+1}}_{k,A_{l+1}}(S_{l+1})}&\lesssim R^{O(\de)} C^{c_l} (\log R)^{3a_l}\sum_{S_{l}}\|\Sq(\vf_{S_l})\|^{p_{l+1}}_{\BL^{p_{l+1}}_{k,2A_l}(S_l)},\\ \label{est0.2}
    \sum_{S_l}\|\Sq\vf_{S_l}\|^2_2\lesssim R^{O(\de)}& C^{c_l} D_l {\poly}(d)^{a_l}R^{-\de a_l}\sum_{S_{l+1}}\|\Sq\vf_{S_{l+1}}\|_2^2,\\
    \label{est0.3}
    \|\Sq\vf_{S_l}\|^2_2\lesssim  R^{O(\de)} C^{c_l} \Big( R^{\de a_l} &\Big)^{-\frac{n-(l+1)}{2}} D_l^{-l} R^{-\de a_l}\|\Sq\vf_{S_{l+1}}\|_2^2,\ \ \forall S_{l}<S_{l+1}.
\end{align}

}
\end{algorithm}

\begin{proof}

Let us first discuss the algorithm from dimension $l+1$ to $l$.
Suppose that we have finished the setup at dimension $l+1$. That is: at dimension $l+1$, we have a scale $r_{l+1}$ ($M^2R^{\e/10n}\le r_{l+1}\le R$) and a number $A_{l+1}$. Also, we have a set of cells $\cs_{l+1}=\{S_{l+1}\}$, so that each $S_{l+1}$ is contained in a rescaled $r_{l+1}$-ball $\cp_{S_{l+1}}$. Assume that for each cell there is an $(l+1)$-dimensional transverse complete intersection $\bZ_{S_{l+1}}$ and a function $\vf_{S_{l+1}}$. Each component of $\vf_{S_{l+1}}$, $f_{j,S_{l+1}}$, has Fourier transform supported in $N_{R^{\beta}r_{l+1}^{-1}}(\Ga_j(\si))$, and is concentrated on wave packets that are $r_{l+1}^{-1/2}$-tangent to $\bZ_{S_{l+1}}$ in $\cp_{S_{l+1}}$.

Apply Algorithm \ref{1algorithm} to them. We set
\begin{align}
    &c_l:=\#\{1\le i\le s:~\textup{STATE}(i)=\textup{cell}\}, \\ 
    &a_l:=\#\{1\le i\le s:~\textup{STATE}(i)\neq\textup{cell}\},
\end{align}
as in \eqref{it1.11} and \eqref{it1.12}. Then the new scale is given by
$$r_l=r_{l+1}2^{-c_l}R^{-\de a_l}.$$
We also set
\begin{align}
\nonumber
    D_l:=d^{c_l}, \hspace{.5cm}
    A_l:=A_{l+1}/2^{a_l+1}.
\end{align} 
Regarding to  STATE($s$), we have the following two scenarios:
\medskip

\noindent
\textit{Scenario 1} (STATE($s$)=tang).
We obtain a set of tangent cells $\cs_{l}=\{S_{l}\}$, each of which is contained in a rescaled $r_{l}$-ball $\cp_{S_{l}}$. For each cell $S_l$, we obtain an $l$-dimensional transverse complete intersection $\bZ_{S_l}$ and a function $\vf_{S_l}$. Each component of $\vf_{S_l}$, $f_{j,S_l}$, has Fourier transform supported in $N_{R^{\beta}r_l^{-1}}(\Ga_j(\si))$, and is concentrated on wave packets that are $r_l^{-1/2}$-tangent to $\bZ_{S_l}$ in $\cp_{S_l}$.

The most important estimates we obtain are:
\begin{align}\label{est1.1}
    \sum_{S_{l+1}}\|\Sq\vf_{S_{l+1}}\|^{p_{l+1}}_{\BL^{p_{l+1}}_{k,A_{l+1}}(S_{l+1})}&\lesssim R^{O(\de)} C^{c_l} (\log R)^{3a_l}\sum_{S_{l}}\|\Sq(\vf_{S_l})\|^{p_{l+1}}_{\BL^{p_{l+1}}_{k,2A_l}(S_l)},\\ \label{est1.2}
    \sum_{S_l}\|\Sq\vf_{S_l}\|^2_2\lesssim R^{O(\de)}& C^{c_l} D_l {\poly}(d)^{a_l}R^{-\de a_l}\sum_{S_{l+1}}\|\Sq\vf_{S_{l+1}}\|_2^2.
\end{align}
Also, for each $S_{l}<S_{l+1}$ one has
\begin{equation}\label{est1.3}
    \|\Sq\vf_{S_l}\|^2_2\lesssim R^{O(\de)} C^{c_l} \Big( R^{\de a_l} \Big)^{-\frac{n-(l+1)}{2}} D_l^{-l} R^{-\de a_l}\|\Sq\vf_{S_{l+1}}\|_2^2.
\end{equation}
Here, \eqref{est1.1} is obtained by iterating \eqref{itcellular3}, \eqref{ittransverse3},  \eqref{ittangent3}; \eqref{est1.2} is obtained by iterating \eqref{itcellular1}, \eqref{ittransverse1}, \eqref{ittangent1}; \eqref{est1.3} is obtained by iterating \eqref{itcellular2}, \eqref{ittransverse2},  \eqref{ittangent2}.

\medskip
\noindent
\textit{Scenario 2} (STATE($s$)$\neq$tang).
In this scenario, $r_l\sim M^2R^{\e/10n}$, and the algorithm stops.
We obtain our final collection of cells $\co=\{O\}$, each of which is contained in a rescaled $r_{l}$-ball $\cp_{O}$. For each cell $O$, we obtain an $l$-dimensional transverse complete intersection $\bZ_{O}$ and a function $\vf_{O}$. Each component of $\vf_{O}$, $f_{j,O}$, has Fourier transform supported in $N_{R^{\beta}r_l^{-1}}(\Ga_j(\si))$, and is concentrated on wave packets that are $r_l^{-1/2}$-tangent to $\bZ_{O}$ in $\cp_{O}$.

Similarly, we obtain:
\begin{align}\label{est2.1}
    \sum_{S_{l+1}}\|\Sq\vf_{S_{l+1}}\|^{p_{l+1}}_{\BL^{p_{l+1}}_{k,A_{l+1}}(S_{l+1})}\lesssim & R^{O(\de)} C^{c_l} (\log R)^{3a_l}\sum_{O}\|\Sq(\vf_{O})\|^{p_{l+1}}_{\BL^{p_{l+1}}_{k,A_{l+1}}(O)},\\ \label{est2.2}
    \sum_{O}\|\Sq\vf_{O}\|^2_2\lesssim R^{O(\de)}& C^{c_l}  D_l {\poly}(d)^{a_l}R^{-\de a_l}\sum_{S_{l+1}}\|\Sq\vf_{S_{l+1}}\|_2^2.
\end{align}
For each $O<S_{l+1}$, one has
\begin{equation}\label{est2.3}
    \|\Sq\vf_{O}\|^2_2\lesssim R^{O(\de)} C^{c_l} \Big( R^{\de a_l}\Big)^{-\frac{n-(l+1)}{2}} D_l^{-l} {\poly}(d)^{a_l}R^{-\de a_l}\|\Sq\vf_{S_{l+1}}\|_2^2.
\end{equation}

\medskip

Now we iterate the above argument.
Let us begin with the setup at dimension $n$.
We begin with a scale $r_n=R$, a single cell $S_n=\cp_R$
and the function $\vf_{S_n}=\vf$. Each component of $\vf$, $f_j$, has Fourier support in $N_{R^{\beta}R^{-1}}\Ga_j(\si)$. We choose $\bZ_{S_n}=\R^n$ to be the $n$-dimensional transverse complete intersection. 
If we are in \textit{Scenario 2}, we end our algorithm and obtain cells $\co$ and functions $\{\vf_O\}_{O\in\co}$. If we are in \textit{Scenario 1}, then we obtain cells $\cs_{n-1}$,  functions $\{\vf_{S_{n-1}}\}_{S_{n-1}\in\cs_{n-1}}$.  We can continue the same reasoning at dimension $n-1$ for $\cs_{n-1}$. There are still two scenarios: \textit{Scenario 1} and \textit{Scenario 2}. We continue when it is in \textit{Scenario 1}; stop when it is in \textit{Scenario 2}. The iteration will finally stop when we encounter \textit{Scenario 2}, since we cannot go below dimension $k$.

\smallskip

Suppose the second algorithm ends at dimension $m$, then $m$ satisfies $k\le m\le n$ (since we are considering the $k$-broad norm).
Let $\cs_n,\cs_{n-1}\cdots,\cs_{m},\co$ be the cells we obtain from the second algorithm.
We also let 
\begin{equation}
    R=r_n\ge r_{n-1}\ge \cdots r_m\ge r_{m-1}\sim M^2 R^{\e/10n}.
\end{equation}
be the scales of these cells. We also see \eqref{est0.1}, \eqref{est0.2}, \eqref{est0.3} are satisfied.  
\end{proof}

Before stating the results, we introduce a notation to simplify the calculation.

\begin{remark}\label{pretend}

\rm
In the following computations, we will use $`` A\lessapprox B"$ to denote 
\begin{equation}
\nonumber
    A\lesssim R^{O(\de)} \prod_{l=m-1}^n D_l^{O(\de)}B.
\end{equation}
With this notation, we can ignore factors like $R^{O(\de)}$, $(\log R)^{3a_l}$, $C^{c_l}$ in \eqref{est1.1}--\eqref{est1.3}, since $(\log R)^{3a_l}\!\le\! (\log R)^{2\de^{-1}}\!\!\!\lesssim \!R^{\de}$, $C^{c_l}\lesssim \! (d^{O(\de)})^{c_l}=D_l^{O(\de)}$ if $d$ is sufficiently large. Also we can pretend $r_{l+1}\approx r_l R^{-\de a_l}$, as $2^{c_l}=D_l^{O(\de)}$ if $d$ is sufficiently large.
\end{remark}

\begin{lemma}
For $m\le l\le n$, we have
\begin{equation}\label{iterationineq}
    \|\Sq\vf\|_{\BL^{p_n}_{k,A_n}(\cp_R)}\lessapprox  M(l) \|\Sq \vf\|_{L^2(\om_{\cp_R})}^{1-\beta_l}\big( \sum_{S_l}\|\Sq \vf_{S_l}\|^{p_l}_{\BL^{p_l}_{k,A_l}(S_l)} \big)^{\frac{\beta_l}{p_l}}
\end{equation} 
where 
\begin{equation}\label{ml}
    M(l)=\prod_{i=l}^{n-1} (D_i R^{-\de a_{i}})^{\frac{\beta_{i+1}-\be_l}{2}}.
\end{equation}
\end{lemma}

\begin{proof}
We induct on $l$. When $l=n$, it holds. Suppose \eqref{iterationineq} holds for $l+1$, i.e.
\begin{equation}\label{iterationineq1}
    \|\Sq\vf\|_{\BL^p_{k,A_n}(\cp_R)}\lessapprox M(l+1) \|\Sq \vf_{\cp_R}\|_{L^2(\om_{\cp_R})}^{1-\beta_{l+1}}\big( \sum_{S_{l+1}}\|\Sq \vf_{S_{l+1}}\|^{p_{l+1}}_{\BL^{p_{l+1}}_{k,A_{l+1}}(S_{l+1})} \big)^{\frac{\beta_{l+1}}{p_{l+1}}}. 
\end{equation} 
From \eqref{est1.1}, we have
\begin{equation}
    \big(\sum_{S_{l+1}}\|\Sq \vf_{S_{l+1}}\|^{p_{l+1}}_{\BL^{p_{l+1}}_{k,A_{l+1}}(S_{l+1})} \big)^{\frac{1}{p_{l+1}}}\lessapprox \big(\sum_{S_{l}}\|\Sq \vf_{S_{l}}\|^{p_{l+1}}_{\BL^{p_{l+1}}_{k,2A_l}(S_{l})} \big)^{\frac{1}{p_{l+1}}}.
\end{equation}
By H\"older's inequality for the broad norm (see \eqref{broad-holder}), we get
\begin{equation}
    \big(\sum_{S_{l}}\|\Sq \vf_{S_{l}}\|^{p_{l+1}}_{\BL^{p_{l+1}}_{k,2A_l}\!(S_{l})} \big)^{\frac{1}{p_{l+1}}}\!\!\lesssim\! \big(\sum_{S_{l}}\|\Sq \vf_{S_{l}}\|^{2}_{L^{2}_{k,A_l}\!(S_{l})} \big)^{\frac{1-\al_l}{2}}\!\! \big(\sum_{S_{l}}\|\Sq \vf_{S_{l}}\|^{p_{l}}_{\BL^{p_{l}}_{k,A_l}\!(S_{l})} \big)^{\frac{\al_l}{p_{l}}}.
\end{equation}
From \eqref{est1.2}, one has
\begin{equation}
    \sum_{S_{l}}\|\Sq \vf_{S_{l}}\|^{2}_{L^{2}_k(S_{l})}\lessapprox \prod_{i=l}^{n-1}D_l R^{-\de a_l}\|\Sq \vf\|_2^2.
\end{equation}
Plugging the above estimates back to \eqref{iterationineq1}, we get that 
\begin{align*}
    \|\Sq\vf\|_{\BL^{p_n}_{k,A_n}(\cp_R)}\lessapprox \,& M(l+1) \|\Sq \vf\|_{L^2(\om_{\cp_R})}^{1-\beta_{l+1}}
    \big(\prod_{i=l}^{n-1}D_l R^{-\de a_l}\|\Sq \vf\|_{L^2(\om_{\cp_R})}^2\big)^{\be_{l+1}\frac{1-\al_l}{2}}\\
    &\cdot\big( \sum_{S_{l}}\|\Sq \vf_{S_{l}}\|^{p_{l}}_{\BL^{p_{l}}_{k,A_l}(S_{l})} \big)^{\frac{\beta_{l+1}\al_l}{p_{l}}}
    \\
    = M(l+1)\big(&\prod_{i=l}^{n-1}D_l R^{-\de a_l}\big)^{\frac{\be_{l+1}-\be_l}{2}}
    \|\Sq \vf\|_{L^2(\om_{\cp_R})}^{1-\beta_{l}}
    \big( \sum_{S_{l}}\|\Sq \vf_{S_{l}}\|^{p_{l}}_{\BL^{p_{l}}_{k,A_l}(S_{l})} \big)^{\frac{\beta_l}{p_{l}}}.
\end{align*}
So it suffices to verify
$$ M(l)=M(l+1)\big(\prod_{i=l}^{n-1}D_l R^{-\de a_l}\big)^{\frac{\be_{l+1}-\be_l}{2}}, $$
which is easy to check.
\end{proof}

The next lemma concerns the estimates for cells $\co=\{O\}$ that are at the smallest scale $r_{m-1}\sim M^2 R^{\e/10n}$.
\begin{lemma}[Small cells]
Recall that $\co$ is the set cells at scale $r_{m-1}\sim M^{2} R^{\e/10n}$. For each $O\in\co$, we have
\begin{equation}\label{smallradius}
    \|\Sq \vf_{O}\|_{\BL^{p}_{k,A_m}(O)}\lesssim M^{(n+1)(\frac{1}{p}-\frac{1}{2})}\|\Sq \vf_{O}\|_2.
\end{equation}
\end{lemma}

\begin{proof}
Since $\|\Sq \vf_{O}\|_{\BL^{p}_{k,A_m}}\lesssim \|\Sq \vf_{O}\|_p$, it suffices to prove
\begin{equation}\label{smallradius1}
     \|\Sq \vf_{O}\|_p\lesssim M^{(n+1)(\frac{1}{p}-\frac{1}{2})}\|\Sq \vf_{O}\|_2.
\end{equation}
Since
\begin{equation}
    \|\Sq\vf_{O}\|_{p}=\Big(\int\Big( \sum_{j=1}^R |f_{j,O}|^2\Big)^{p/2}\Big)^{1/p},
\end{equation}
by H\"older's inequality, it suffices to prove \eqref{smallradius1} when $p=2$ and $p=\infty$. In the case $p=2$, there is nothing to prove.

Let us consider the case for $p=\infty$. We note that $\wh{f}_{j,O_s}$ is supported in a slab of dimensions $M^{-1}\times\cdots\times M^{-1}\times M^{-2}$, so via Bernstein's inequality,
\begin{equation}
\label{bernstein}
    \|f_{j,O}\|_\infty\lesssim M^{-\frac{n+1}{2}}\|f_{j,O_s}\|_2.
\end{equation}
Now we have the estimate
\begin{equation}
    \|\Sq\vf_{O}\|_\infty=\sup_{x}\Big( \sum_{j=1}^R |f_{j,O}(x)|^2\Big)^{1/2}\le \Big( \sum_{j=1}^R \sup_{x}|f_{j,O}(x)|^2\Big)^{1/2},
\end{equation}
which, via the Berstein's estimate \eqref{bernstein}, is bounded from above by
\begin{equation}
    M^{-\frac{n+1}{2}}\Big( \sum_{j=1}^R \|f_{j,O_s}\|_2^2\Big)^{1/2}=M^{-\frac{n+1}{2}}\|\Sq\vf_{O_s}\|_2.
\end{equation}
This is the desired estimate when $p=\infty$.
\end{proof}

Combining the above two lemmas, we can get an upper bound for $\|\Sq \vf\|_{\BL^p_k(\cp_R)}$.

\begin{proposition}
We have the estimate
\begin{align}\label{mixnorm}
   \|\Sq\vf\|_{\BL^p_{k,A_n}\!(\cp_R)}\!\lessapprox & M^{\frac{2\be_m}{p_m}+(n+1)(\frac{1}{p_n}-\frac{1}{2})}R^{-\frac{1}{p_n}}\!\prod_{i=m}^{n-1}r_i^{\frac{\be_{i+1}-\be_i}{2}}\!\!\!\!\prod_{i=m-1}^{n-1}\!\!\!\!D_i^{\frac{\be_{i+1}}{2}-(\frac{1}{2}-\frac{1}{p_n})}\\ \nonumber
   &\|\Sq\vf\|_{L^2(\om_{\cp_R})}^{\frac{2}{p_n}}\max_{O}\|\Sq \vf_{O}\|^{1-\frac{2}{p_n}}_{2}.
\end{align}
\end{proposition}

\begin{proof}
Recall our convention $\co=\cs_{m-1}$.
Combining \eqref{iterationineq} with $l=m$, \eqref{est0.1} with $l=m-1$ and \eqref{smallradius}, we get
\begin{align}\label{eee}
    \|\Sq\vf\|_{\BL^p_{k,A_n}(\cp_R)}&\lessapprox M(m) M^{(n+1)(\frac{1}{p_n}-\frac{1}{2})} \|\Sq \vf\|_{L^2(\om_{\cp_R})}^{1-\beta_m}\big( \sum_{O}\|\Sq \vf_{O}\|^{p_m}_{2} \big)^{\frac{\beta_m}{p_m}}.
\end{align}
Iterating \eqref{est0.2} from $l=m-1$ to $l=n-1$, we get
\begin{equation*}
    \sum_{O}\|\Sq \vf_{O}\|^{2}_{2}\lessapprox \big(\prod_{i=m-1}^{n-1}D_i R^{-\de a_{i}} \big)\|\Sq\vf\|_{L^2(\om_{\cp_R})}^2.
\end{equation*}
Plugging it into \eqref{eee} and recalling \eqref{ml}, we have
\begin{align*}
    &\|\Sq\vf\|_{\BL^p_{k,A_n}(\cp_R)}
    \lessapprox  M(m) M^{(n+1)(\frac{1}{p_n}-\frac{1}{2})}\|\Sq \vf\|_{L^2(\om_{\cp_R})}^{1-\beta_m}\big( \sum_{O}\|\Sq \vf_{O}\|^{2}_{2} \big)^{\frac{\beta_m}{p_m}}\max_{O}\|\Sq \vf_{O}\|^{1-\frac{2}{p_n}}_{2}\\
    &\lessapprox  \prod_{i=m}^{n-1} (D_i R^{-\de a_{i}})^{\frac{\beta_{i+1}-\be_m}{2}} \prod_{i=m-1}^{n-1} (D_i R^{-\de a_{i}})^{\frac{\beta_m}{p_m}} M^{(n+1)(\frac{1}{p_n}-\frac{1}{2})}\|\Sq\vf\|_{L^2(\om_{\cp_R})}^{\frac{2}{p_n}} \max_{O}\|\Sq \vf_{O}\|^{1-\frac{2}{p_n}}_{2}.
\end{align*}

We simplify the formula:
\begin{align*}
    \prod_{i=m}^{n-1} (D_i R^{-\de a_{i}})^{\frac{\beta_{i+1}-\be_m}{2}} &\prod_{i=m-1}^{n-1} (D_i R^{-\de a_{i}})^{\frac{\beta_m}{p_m}}\lessapprox \prod_{i=m}^{n-1}\Big( \frac{r_i D_i}{r_{i+1}} \Big)^{\frac{\beta_{i+1}-\beta_m}{2}}\prod_{i=m-1}^{n-1}\Big( \frac{r_i D_i}{r_{i+1}} \Big)^{\frac{\be_m}{p_m}}\\
    &=r_{m-1}^{\frac{\be_m}{p_m}}\prod_{i=m}^{n-1}r_i^{\frac{\be_{i+1}-\be_i}{2}}r_n^{-(\frac{\be_n-\be_m}{2}+\frac{\be_m}{p_m})}\prod_{i=m-1}^{n-1}D_i^{\frac{\be_{i+1}-\be_m}{2}+\frac{\be_m}{p_m}}
\end{align*}
Recall $r_n=R, r_{m-1}\sim M^2 R^{\e/10n}, \be_n=1, \be_i\big( \frac{1}{2}-\frac{1}{p_i} \big)=\frac{1}{2}-\frac{1}{p_n}$. The above equals
\begin{equation}
    M^{\frac{2\be_m}{p_m}}R^{-\frac{1}{p_n}}\prod_{i=m}^{n-1}r_i^{\frac{\be_{i+1}-\be_i}{2}}\prod_{i=m-1}^{n-1}D_i^{\frac{\be_{i+1}}{2}-(\frac{1}{2}-\frac{1}{p_n})}.
\end{equation}
Combining the estimates above, we proved \eqref{mixnorm}.
\end{proof}

\section{Estimate the functions associated to the smallest cells}\label{section8}
We estimate $\max_{O}\|\Sq \vf_{O}\|_2$ in this section.
First, we discuss the nested polynomial Wolff estimate.

\subsection{Nested polynomial Wolff}

Recall the relation between tubes in Definition \ref{tuberelation} .
We need the following result of Zahl:
\begin{lemma}[\cite{zahl2021new} Lemma 2.11]\label{lemzahl}
Fix $r_n\ge r_{n-1}\ge \cdots\ge r_l>0$ and $\rho_n\ge \rho_{n-1}\ge \cdots\ge \rho_l>0$ so that $1\ge \frac{\rho_l}{r_l}\ge \frac{\rho_{l+1}}{r_{l+1}}\ge \cdots \ge \frac{\rho_n}{r_n}$. Let $S_n\supset S_{n-1}\supset \cdots \supset S_{l}$ be semi-algebraic sets of complexity at most $E$ such that for each $i$, $S_i$ is $i$-dimensional and contained in $B_{r_i}(x_i)$. We recursively define another sequence of sets $\wt S_i$ $(l\le i\le n)$.

We define $\wt S_l:=N_{2\rho_l}(S_l)$. For each $i=l,\cdots, n-1$, define 
\begin{equation}
\wt S_{i+1}:= N_{2\rho_{i+1}}(S_{i+1})\cap \bigcup_{\begin{subarray}{c}
     T\textup{~a~}\rho_i\times r_i\textup{~tube}  \\
     T\subset \wt S_{i} 
\end{subarray} } \textup{Fat}_{\frac{r_{i+1}}{r_i}}(T).
\end{equation}
Here $\textup{Fat}_A T$ is the $A$-dilation of $T$ with respect to the center of $T$.

Then we have
\begin{equation}\label{estmeasure}
    |\wt S_n|\le C(n,E,\e')r_n^{\e'}r_n^n \prod_{i=l}^{n-1}\frac{\rho_i}{r_i},
\end{equation}
for any $\e'>0$.
\end{lemma}

\begin{remark}
\rm

The estimate here is for the measure of $\wt S_n$, while in \cite{zahl2021new} is for the number of tubes. But actually \eqref{estmeasure} is already proved in (2.33) in \cite{zahl2021new} .
\end{remark}

We will actually apply the following rescaled version:

\begin{proposition}[Nested polynomial Wolff on a small cap]\label{smallcapwolff}
Fix $M>1$.
Fix $r_n\ge r_{n-1}\ge\cdots\ge r_l>0$ and $\rho_n\ge \rho_{n-1}\ge\cdots\ge \rho_l>0$ so that $M^{-1}\ge\frac{\rho_l}{r_l}\ge \frac{\rho_{l+1}}{r_{l+1}}\ge\cdots \ge \frac{\rho_n}{r_n}$. Let $S_n\supset S_{n-1}\supset \cdots \supset S_{l}$ be semi-algebraic sets of complexity at most $E$ such that for each $i$, $S_i$ is $i$-dimensional and contained in $\cp_{r_i}(x_i)$ (a rescaled ball of dimensions $M^{-1}r_i\times\cdots\times M^{-1}r_i\times r_i$ whose long side points to the direction of $\vec e_n$). We recursively define another sequence of sets $\wt S_i$ $(l\le i\le n)$.

We define $\wt S_k:=N_{2\rho_l}(S_l)$, and for each $i=l,\cdots, n-1$, define
\begin{equation}\label{smallcapwolff1}
\wt S_{i+1}:= N_{2\rho_{i+1}}(S_{i+1})\cap \bigcup_{\begin{subarray}{c}
     T\textup{~a~}\rho_i\times r_i\textup{~tube}  \\
     T\subset \wt S_{i} 
\end{subarray} } \textup{Fat}_{\frac{r_{i+1}}{r_i}}(T).
\end{equation}
Then we have
\begin{equation}\label{polynomialwolff}
    |\wt S_n|\le C(n,E,\e')r_n^{\e'}M^{-(l-1)}r_n^n \prod_{i=l}^{n-1}\frac{\rho_i}{r_i},
\end{equation}
for any $\e'>0$.

\end{proposition}

\begin{proof}
Let us see how Lemma \ref{lemzahl} implies Proposition \ref{smallcapwolff}. We rescale by factor $M^{-1}$ in the $\vec e_n$ direction, then each $\cp_{r_i}(x_i)$ becomes a ball $B_{M^{-1}r_i}(x_i)$. Denote by $S'_i$ the set $S_i$ after rescaling.
We consider the following sets.

Let $\wt S'_l=N_{2\rho_l}(S'_l)$. For each $i=l,\cdots, n-1$, let
\begin{equation}
\wt S'_{i+1}= N_{2\rho_{i+1}}(S'_{i+1})\cap \bigcup_{\begin{subarray}{c}
     T\textup{~a~}\rho_i\times M^{-1}r_i\textup{~tube}  \\
     T\subset \wt S'_{i} 
\end{subarray} } \textup{Fat}_{\frac{r_{i+1}}{r_i}}(T).
\end{equation}
One sees that after rescaling, $\wt S_i$ becomes a subset of $\wt S'_i$.
Now we apply Lemma \ref{lemzahl} to $\{M^{-1}r_i\}, \{\rho_i\}, \{S'_i\}$ so that
\begin{equation}
    |\wt S'_n|\le C(n,E,\e')(M^{-1}r_n)^{n+\e'} \prod_{j=l}^{n-1}\frac{\rho_j}{M^{-1}r_j}\le C(n,E,\e')M^{-l}r_n^{\e'}r_n^n \prod_{j=l}^{n-1}\frac{\rho_j}{r_j}.
\end{equation}
Finally, \eqref{polynomialwolff} follows from $|\wt S_n|\le M|\wt S'_n|$. 
\end{proof}

Next, we discuss how to use this proposition in our setting. Recall \eqref{obtainradius}. After the second algorithm, we obtain a sequence of radius $$R=r_n\ge r_{n-1}\ge \cdots\ge r_m\ge r_{m-1}\sim M^2 R^{\e/10n}$$ and the cells $$\cs_n(=\cp_R),\ \cs_{n-1},\ \cdots,\ \cs_m,\ \cs_{m-1}(=\co).$$ 
We set $\rho_i=r_i^{1/2}R^{\dc}$ ($m-1\le i\le n$). From the second algorithm, for any $m\le l\le n$ and $S_l\in\cs_l$, we have a sequence of nested semi-algebraic sets $S_n\supset S_{n-1}\cdots \supset S_{l}$ such that each $S_i$ is contained in a rescaled $r_i$-ball $\cp_{S_i}$. For each $S_i$, there is a set of $r_i$-tubes $\T_{S_i}$ that are $r_i^{-1/2}$-tangent to $S_i$ in $\cp_{S_i}$ (recall Definition \ref{tangenttube}). 

Currently we make the following assumption on these tubes: 
\textit{For any $l\le i\le t\le n$ and any $T_t\in\T_{S_t}$, there exists a $T_i\in\T_{S_i}$ such that $T_i<T_t$}. Our goal is to give an upper bound on $|\bigcup_{T_n\in\T_{S_n}} T_n|$.

If we define $\wt S_i$ $(l\le i\le n)$ as in \eqref{smallcapwolff1}, then we claim that
\begin{equation}
    \bigcup_{T_i\in\T_{S_{i}}} T_i \subset \wt S_i. 
\end{equation}
We prove the claim by induction on $i$. When $i=l$, it is just by definition. If step $i$ is proved, consider $i+1$. First of all, we have $\bigcup_{T_{i+1}\in\T_{S_{i+1}}} T_{i+1}\subset N_{2\rho_{i+1}}(S_{i+1})$ since $\T_{S_{i+1}}$ are $r_{i+1}^{-1/2}$-tangent to $S_{i+1}$. 
It remains to show
\begin{equation}\label{tubebelong}
    \bigcup_{T_{i+1}\in\T_{S_{i+1}}} T_{i+1} \subset \bigcup_{\begin{subarray}{c}
     T\textup{~a~}\rho_i\times r_i\textup{~tube}  \\
     T\subset \wt S_{i} 
\end{subarray} } \textup{Fat}_{\frac{r_{i+1}}{r_i}}(T). 
\end{equation} 
For each $T_{i+1}\in\T_{S_{i+1}}$, there exists a $T_i\in\T_{S_i}$ such that $T_i<T_{i+1}$. We observe that $T_i<T_{i+1}$ and $\rho_i\frac{r_{i+1}}{r_i}\ge\rho_{i+1}$ imply $T_{i+1}\subset \textup{Fat}_{\frac{r_{i+1}}{r_i}}(T_i)$. By induction $T_i\subset \wt S_i$, we proved \eqref{tubebelong}.

From \eqref{polynomialwolff}, we have
\begin{equation}\label{polywolffineq}
    \Big|\bigcup_{T_n\in\T_{S_n}}T_n\Big|\le C(n,E,\de)R^{2\de}M^{-(l-1)}R^n \prod_{i=l}^{n-1}r_i^{-1/2}.
\end{equation}

\subsection{Estimate \texorpdfstring{$\max_O \|\Sq\vf_O\|_2$}{Lg}}

\begin{lemma}\label{lem1} For $m\le l\le n$, each $O<S_l$, we have
\begin{equation}\label{final1}
    \|\Sq\vf_O\|_2^2\lessapprox r_{m-1}^{\frac{n-m}{2}} r_l^{-\frac{n-l}{2}}\prod_{i=m}^{l-1}r_i^{-\frac{1}{2}}\prod_{i=m-1}^{l-1}D_i^{-i}R^{-\de a_i}\|\Sq\vf_{S_l}\|_2^2.
\end{equation}
\end{lemma}

\begin{proof}
By \eqref{est2.3} and \eqref{est1.3}, and by Remark \ref{pretend} that $r_{i+1}\approx r_i R^{\de a_i}$, we have
\begin{align}
\nonumber
    \|\Sq\vf_O\|_2^2\lessapprox &\big( \frac{r_m}{r_{m-1}} \big)^{-\frac{n-m}{2}}D_{m-1}^{-(m-1)}R^{-\de a_{m-1}} \|\Sq\vf_{S_m}\|_2^2\ \ \ \textup{for~}O<S_{m},\\ \nonumber
    \|\Sq\vf_{S_i}\|_2^2\lessapprox \big( \frac{r_{i+1}}{r_{i}} &\big)^{-\frac{n-(i+1)}{2}}D_{i}^{-i}R^{-\de a_{i}} \|\Sq\vf_{S_{i+1}}\|_2^2\ \ \ \textup{for~}S_i<S_{i+1},\ \ m\le i\le n-1.
\end{align}
Combining these two estimates, we prove the result.
\end{proof}

To state out next lemma, we need to introduce some new notations. 

\begin{definition}\label{defnesttube}
For $l$ and $t$ satisfying $m\le l< t\le n$ and $S_l\in\cs_l$, we are going to define $\vf_{S_l,t}$. Let $S_t\in \cs_t$ be the ancestor of $S_l$: $S_l<S_t$, we define 
\begin{equation}
\nonumber
    \T_{S_t,S_l}:=\{T_t\in\T_{S_t}:\ \exists\  T_i\in\T_{S_i}\ (l\le i\le t-1), \textup{~so~that~}\ T_l<T_{l+1}<\cdots<T_t\};
\end{equation}
\begin{equation}
\nonumber
    \vf_{S_t,S_l}:=\sum_{T_t\in \T_{S_t,S_l}} (\vf_{S_t})_{T_t}.
\end{equation}
For simplicity, we denote  $\vf^\#_{S_l}=\vf_{S_n,S_l}$, $\vf_{S_l}=\vf_{S_l,S_l}$.
\end{definition}

Let us digest this definition. Since $\vf_{S_t}$ is concentrated on wave packets from $\T_{S_t}$, if we ignore the rapidly decaying term, then
\begin{equation}
\nonumber
    \vf_{S_t}=\sum_{T_t\in\T_{S_t}} (\vf_{S_t})_{T_t}.
\end{equation}
Hence we see that $\vf_{S_t,S_l}$ is the sum over a subset of wave packets of $\vf_{S_t}$. These wave packets are related to $\T_{S_l}$.

\begin{lemma}\label{lem2} For $m\le l\le n$ and any $S_l\in\cs_l$, we have
\begin{equation}\label{final2}
    \|\Sq\vf_{S_l}\|_2^2\lessapprox r_l^{\frac{n-l}{2}}\prod_{i=l}^{n-1}r_i^{-\frac{1}{2}}\big(\frac{r_l}{R}\big)\|\Sq\vf_{S_l}^{\#}\|_2^2.
\end{equation}

\end{lemma}

\begin{proof}
We claim that for $m\!\le\! l\!\le\! t\!\le\! n$, $S_t\!\in\!\cs_t, S_{t+1}\!\in\!\cs_{t+1}$ and $S_t<S_{t+1}$, we have
\begin{equation}\label{final2key}
    \|\Sq\vf_ {S_t,S_l}\|_2^2\lessapprox \big(\frac{r_{t+1}}{r_t}\big) ^{-\frac{n-(t+1)}{2}} R^{-\de a_t} \|\Sq\vf_ {S_{t+1},S_l}\|_2^2.
\end{equation}
Using \eqref{final2key} and noting $R\approx r_l \prod_{i=l}^{n-1}R^{\de a_i}$ by Remark \ref{pretend}, we can prove \eqref{final2}. So our main goal is to prove 
\eqref{final2key}.

By the Algorithm \ref{1algorithm}, we have the intermediate cells between $S_l$ and $S_{l+1}$:
\begin{equation}
\nonumber
    S_l=O_s<\cdots<O_u<O_{u-1}<\cdots<O_0=S_{l+1},
\end{equation}
and the iteration formula
\begin{equation}
\nonumber
    \vf_{O_u}=\sum_{T_u\in\T_{O_u}}(\vf_{O_{u-1}})_{T_u}.
\end{equation}
Similar to Definition \ref{defnesttube}, for $S_l<O_{u}$, we define the tubes 
$$ \T_{O_{u},S_l}:=\{T_{u}\in\T_{O_{u}}:\ \exists\  T\in\T_{S_l} \textup{~so~that~}\ T<T_{u}\}, $$
and function
$$ \vf_{O_{u},S_l}:=\sum_{T_u\in\T_{O_{u},S_l}}(\vf_{O_{u}})_{T_{u}}. $$
We are going to show 
\begin{equation}\label{com}
    \vf_{O_u,S_l}=\sum_{T_u\in \T_{O_u,S_l}}(\vf_{O_{u-1},S_l})_{T_u}+\rap(r_l)\|\Sq\vf\|_2.
\end{equation}

First, we consider when $u=s$, which is
\begin{equation}\label{com0}
    \vf_{S_l}=\sum_{T\in \T_{S_l}}(\vf_{O_{s-1},S_l})_{T}+\rap(r_l)\|\Sq\vf\|_2.
\end{equation}
Note the iteration formula
\begin{equation}\label{com1}
    \vf_{S_l}=\sum_{T\in\T_{S_l}}(\vf_{O_{s-1}})_{T}.
\end{equation}
Since $\vf_{O_{s-1}}$ is concentrated on wave packets from $\T_{O_{s-1}}$, we can write
\begin{equation}
\nonumber
    \vf_{O_{s-1}}=\sum_{T_{s-1}\in\T_{O_{s-1},S_l}}(\vf_{O_{s-1}})_{T_{s-1}}+\sum_{T_{s-1}\in\T_{O_{s-1}}\setminus\T_{O_{s-1},S_l}}(\vf_{O_{s-1}})_{T_{s-1}}.
\end{equation}
By definition, the first term is $\vf_{O_{s-1},S_t}$. By Lemma \ref{compare}, we have 
\begin{equation}
\nonumber
    ((\vf_{O_{s-1}})_{T_{s-1}})_T=\rap(r_l)\|\Sq\vf\|_2
\end{equation}
for $T_{s-1}\not\in \T_{O_{s-1},S_l}$ and $T\in\T_{S_l}$. The above two estimates and \eqref{com1} imply \eqref{com0}. The reasoning for \eqref{com} about other $u$ are the same, so we omit the proof.

To get \eqref{final2key}, when STATE($u$)=cell, we use the trivial estimate:
\begin{equation}\label{compare1}
\|\Sq\vf_{O_u,S_l}\|_2^2\lesssim \|\Sq\vf_{O_{u-1},S_l}\|_2^2;
\end{equation}
when STATE($u$)=trans, we use the transverse equidistribution estimate \eqref{TE1} and the same idea in the proof of \eqref{tetrans3} to get:
\begin{equation}\label{compare2}
    \|\Sq\vf_{O_u,S_l}\|_2^2\lesssim R^{O(\dc)} R^{-\de\frac{n-(l+1)}{2}} R^{-\de} \|\Sq\vf_{O_{u-1},S_l}\|_2^2.
\end{equation}

Combining \eqref{compare1} and \eqref{compare2} gives \eqref{final2key} when $t=l$. For other $t$, we can proceed in the same way.
\end{proof}

\begin{remark}
\rm

One may compare \eqref{final2key} with \eqref{est1.3} where there is an additional factor $D_l^{-l}$. The proof of \eqref{est1.3} is by iterating \eqref{itcellular2} and \eqref{ittransverse2} (which can be viewed as counterparts of \eqref{compare1} and \eqref{compare2}). However, for the proof of \eqref{final2key}, we don't have the strong estimate \eqref{itcellular2}, instead we only have the trivial bound \eqref{compare1}.
\end{remark}

\begin{lemma}\label{lem3} For $m\le l\le n$,
\begin{equation}\label{final3}
    \|\Sq\vf_{S_l}^{\#}\|_2^2\lessapprox M^{-(l-1)} R^n \prod_{i=l}^{n-1}r_i^{-\frac{1}{2}}\|f\|_\infty^2.
\end{equation}
\end{lemma}

\begin{proof}
First, we remind readers that $S_n=\cp_R$. 
Since $\vf_{S_l}^{\#}=\vf_{S_n,S_l}$ is the function at scale $R$, so the components of $\vf_{S_l}^{\#}$ have finitely overlapping Fourier support. We define the corresponding function $f_{S_l}^{\#}$ by summing over all the components of $\vf_{S_l}^{\#}$:
\begin{equation}
\nonumber
    f_{S_l}^{\#}:=\sum_{j=1}^R f_{j,S_l}^{\#},
\end{equation}
and have the estimate
\begin{equation}
\nonumber
    \|\Sq\vf_{S_l}^{\#}\|_2^2\lesssim\|f_{S_l}^{\#}\|_2^2.
\end{equation}
Note that $f_{S_l}^{\#}$ is concentrated on the wave packets of $\vf_{S_l}^{\#}$, which we denoted by
\begin{equation}\label{defuniontube}
    X:=\bigcup_{T\in\T_{S_n,S_l}}100 T. 
\end{equation} 
By the nested polynomial Wolff axioms \eqref{polywolffineq} and Definition \ref{defnesttube}, we have the estimate
\begin{equation}\label{final3.1}
    |X|\lessapprox M^{-(l-1)} R^n \prod_{i=l}^{n-1}r_i^{-\frac{1}{2}}.
\end{equation}

Our next goal is to show 
\begin{equation}\label{final3.2}
    \|f^{\#}_{S_l}\|_2^2 \lesssim \|\Id_X f\|_2^2.
\end{equation}
If this is true, we get
\begin{equation}
    \|\Sq\vf_{S_l}^{\#}\|_2^2=\|f_{S_l}^{\#}\|_2^2\lesssim \|\Id_X f\|_{2}^2\lesssim |X|\|f\|_{\infty}^2.
\end{equation}
Combining this with \eqref{final3.1}, we proved \eqref{final3}.

It remains to prove \eqref{final3.2}. Let us recall the wave packet decomposition for each component of $\vf^\#_{S_l}$:
\begin{equation}
\nonumber
    \vf_{S_l}^\#=\sum_{T\in\T_{S_n,S_l}} \vf _T=\sum_{\theta}\sum_{T\in\T_\theta\cap\T_{S_n,S_l}} \vf_{\theta}\Id^*_T. 
\end{equation}
Here $\T_{S_n,S_l}$ is a set of $R$-tubes, $\sum_\theta$ is a sum over $R^{-1/2}$-caps, $\T_{\theta}$ is the set of $R$-tubes that point to the direction $c_\theta$. Thus, we have
\begin{equation}
    f_{S_l}^\#= \sum_{j=1}^R\sum_{T\in\T_{S_n,S_l}} (f_{j})_T=\sum_{j=1}^R\sum_{\theta}\sum_{T\in\T_\theta\cap\T_{S_n,S_l}} (f_{j})_{\theta}\Id^*_T.
\end{equation}
Recall the definitions in \eqref{operator-1}, \eqref{vector-f} that 
$$ (f_{j})_{\theta}= \vp_{j,\theta}* f_{j}= \vp_{j,\theta}* m_j* f, $$
where $\vp_{j,\theta}$ is a smooth cut off function at the $R^{-1/2}$-slab $N_{R^{-1/2}}\Ga_j(\theta)$ whose normal direction is $c_\theta$ (recall \eqref{gtau}), and $m_j$ is a smooth cut off function at $N_{R^{-1/2}}(\Ga_j)$ (recall \eqref{kernel-1}). 
Intuitively, $\vp_{j,\theta}\cdot m_j\approx \Id_{N_{R^{-1/2}}\Ga_j(\theta)}$. 

Set $m_{j,\theta}:=\vp_{j,\theta}\cdot m_j$. Let $\theta^*$ be the $R$-tube dual to $\theta$ and passes through the origin and let $\Id^*_{\theta^*}$ be a bump function on $\theta^\ast$, then we have the estimate
\begin{equation}
\nonumber
    |m_{j,\theta}(x)|\lesssim \frac{1}{|\theta^*|}\Id^*_{\theta^*}.
\end{equation}

Now, we can write
\begin{align}
    \nonumber f_{S_l}^\#&=\sum_{j=1}^R\sum_{\theta}\sum_{T\in\T_\theta\cap\T_{S_n,S_l}}   (m_{j,\theta}*f)\Id^*_T\\
    \label{secondterm}&= \sum_{j=1}^R\sum_{\theta}\!\!\sum_{T\in\T_\theta\cap\T_{S_n,S_l}}  \!\!\!\!\!\!\!\!\Big(m_{j,\theta}*\big(f\Id_X\big)\Big)\Id^*_T+ \sum_{j=1}^R\sum_{\theta}\!\!\sum_{T\in\T_\theta\cap\T_{S_n,S_l}} \!\!\!\!\!\!\!\!\!\Big(m_{j,\theta}*\big(f\Id_{X^c}\big)\Big)\Id^*_T.
\end{align} 
We claim the second term in \eqref{secondterm} is negligible. Note that for $T\in\T_\theta\cap\T_{S_n,S_l}$, $m_{j,\theta}*\big(f_{j}\Id_{X^c}\big)\Id^*_T$ is essentially supported in $(\theta^*+X^c)\cap T$. This is an empty set since $X^c\cap (\theta^*+T)\subset X^c \cap 5T=\empty $ by the definition of $X$ in \eqref{defuniontube}. Thus,
\begin{align}
     f_{S_l}^\#= \sum_{j=1}^R\sum_{\theta}\sum_{T\in\T_\theta\cap\T_{S_n,S_l}}  \Big(m_{j,\theta}*\big(f\Id_X\big)\Big)\Id^*_T+\rap(R)\|\Sq(\vf)\|_2.
\end{align} 
Note that for each $j, \theta$, the Fourier support of $\sum_{T\in\T_\theta\cap\T_{S_n,S_l}} \Big((m_{j,\theta}*\big(f\Id_X\big)\Big)\Id^*_T$ is contained in $N_{10R^{-1}}\Ga_j(10\theta)$, and the sets $\{N_{10R^{-1}}\Ga_j(10\theta)\}_{j,\theta}$ are still finitely overlapping, so by Plancherel we have 
\begin{equation}
\nonumber
    \|f_{S_l}^\#\|_2^2\lesssim  \sum_{j=1}^R\sum_{\theta}\|\sum_{T\in\T_\theta\cap\T_{S_n,S_l}}  \Big(m_{j,\theta}*\big(f\Id_X\big)\Big)\Id^*_T\|_2^2.
\end{equation}
Since the tubes $T\in\T_\theta$ are essentially disjoint, we further have 
$$ \|f_{S_l}^\#\|_2^2\lesssim  \sum_{j=1}^R\sum_{\theta}\| m_{j,\theta}*\big(f\Id_X\big)\|_2^2. $$
Now by Plancherel again we obtain
\begin{equation}
\nonumber
    \|f_{S_l}^\#\|_2^2\lesssim \|f\Id_X\|_2^2,
\end{equation}
which gives \eqref{final3.2}.
\end{proof}

Combining Lemma \ref{lem1}, Lemma \ref{lem2} and Lemma \ref{lem3} above, we obtain
\begin{equation}\label{L2Linfty}
    \|\Sq\vf_O\|_2^2\lessapprox R^{\e/2}R^{n-1} M^{n-m+3}M^{-l} \prod_{i=m}^{n-1}r_i^{-1/2}\prod_{i=m-1}^{l-1}D_i^{-i}\prod_{i=l}^{n-1}r_i^{-1/2}\|f\|_\infty^2.
\end{equation}
Next, we are going to use \eqref{L2Linfty} for different $l$ to derive the estimate we need.
Recall the definition of numbers $0\le \ga_k,\cdots ,\ga_n\le 1$ and $1\le \si_{k},\cdots,\si_{n}\le 1$ are given in \eqref{defga} and \eqref{defsi}. Since our range of $l$ is $m\le l\le n$, we need new parameters $\{\ga'_l\}_{l=m}^n,\ \{\si'_l\}_{l=m}^n$ which are a slight modification of the definitions of $\ga_l$'s and $\si_l$'s. We define
\begin{align}
    &\ga'_m:=\sum_{i=k}^m \ga_i,\\
    &\ga'_l:=\ga_l\ \ (m+1\le l\le n),\\
    &\si'_{l}:=\sum_{i=k}^l \ga'_i\ \ (m\le l\le n).
\end{align}
One can checks
\begin{equation}
\nonumber
    \sum_{l=m}^n\ga_l'=1, \hspace{.5cm}\si_l'=\si_l\ \ (m\le l\le n),
\end{equation}
So we still use $\si_l$ instead of $\si'_l$.
\smallskip

Take the power of $\ga'_l$ on both sides of $\eqref{L2Linfty}$ and then multiply them together for all the $m\le l \le n$. After some calculations, we obtain
\begin{proposition}
\begin{equation}\label{maxL2_1}
    \|\Sq\vf_O\|_2^2\lessapprox R^{\e/2} R^{n-1} M^{n-m+3-\sum_{l=m}^n l\ga'_l}\prod_{i=m}^{n-1}r_i^{-\frac{1+\si_i}{2}}\prod_{i=m-1}^{n-1}D_i^{-i(1-\si_i)}\|\Sq\vf\|_\infty^2.
\end{equation}
\end{proposition}

We can rewrite the power of $M$ on the right hand side of \eqref{maxL2_1}. Note that
\begin{equation*}
    \sum_{l=m}^n l\ga'_l=m\sum_{l=m}^n\ga'_l+\!\!\sum_{l=m+1}^n (l-m)\ga'_l=m+\!\!\sum_{i=m+1}^{n} \sum_{l=i}^n\ga'_l=m+\sum_{i=m}^{n-1}(1-\si_i)=n-\sum_{i=m}^{n-1}\si_i.
\end{equation*}
Recall \eqref{solve1}. After summing over all $m\le l\le n-1$ we get
\begin{equation}
    (\frac{1}{2}-\frac{1}{p_n})\sum_{i=m}^{n-1} \si_i=\be_n-\be_m-(n-m)(\frac{1}{2}-\frac{1}{p_n}).
\end{equation}
This implies
\begin{equation}
\label{maxL2-M}
    M^{-\sum_{l=m}^n l\ga'_l}=M^{-n}\prod_{i=m}^{n-1}M^{\si_i}=M^{-n+(\frac{1}{2}-\frac{1}{p_n})^{-1}\big( \be_n-\be_m-(n-m)(\frac{1}{2}-\frac{1}{p_n}) \big)}.
\end{equation}
Therefore, combining \eqref{mixnorm}, \eqref{maxL2_1} and \eqref{maxL2-M}, we obtain the following key estimate:
\begin{equation}\label{finalestimate1}
    \|\Sq\vf\|_{\BL^{p_n}_{k,A_n}(\cp_R)}\lessapprox R^{\e/2} R^{\frac{n-1}{2}-\frac{n}{p}}M^{e(M)} \prod_{l=m}^{n-1}r_l^{X_l}\prod_{l=m-1}^{n-1}D_l^{Y_l}\|\Sq\vf\|_{L^2(\om_{\cp_R})}^{2/p_n}\|f\|_\infty^{1-\frac{2}{p_n}},
\end{equation}
where  
\begin{align*}
    &X_l=\frac{\be_{l+1}-\be_l}{2}-\frac{1+\si_l}{2}\big(\frac{1}{2}-\frac{1}{p_n}\big), \hspace{.5cm} Y_l=\frac{\be_{l+1}}{2}-(1+l(1-\si_l))\big(\frac{1}{2}-\frac{1}{p_n}\big),\\ 
    &e(M)=\frac{2\be_m}{p_m}+(n+1)(\frac{1}{p_n}-\frac{1}{2})
    +\be_n\!-\!\be_m\!-\!(n-m)(\frac{1}{2}-\frac{1}{p_n})\!-\!(m-3)(\frac{1}{2}-\frac{1}{p_n}).
\end{align*}
We can compute that $e(M)=\frac{2n}{p_n}-(n-1)$. Also from \eqref{solve1} and \eqref{solve2}, we have $X_l=Y_l=0$. We can slightly perturb the choice of $\si_l$'s so that $X_l, Y_l\le -C\de$ at the cost of making $p_n$ slightly bigger. Actually, when we are solving $\{\si_l\}$ from \eqref{solve1} and \eqref{solve2}, the equality $``=0"$ should be replaced by $``=-C\de"$. Since it is a small modification, we do not put more details here. 

Now if we want to replace the $p_n$ in \eqref{finalestimate1} by the $p$ in 
\eqref{defpk}, then we need $Y_l\le -C\de$, which can be used to compensate for the implicit factor $\prod_{l=m-1}^n D_l^{O(\de)}$ of ``$\lessapprox$" in \eqref{finalestimate1} (see Remark \ref{pretend}). So we actually obtain
\begin{equation}\label{finalestimate}
    \|\Sq\vf\|_{\BL^{p}_{k,A_n}(\cp_R)}\lesssim R^{\e} R^{\frac{n-1}{2}-\frac{n}{p}}M^{\frac{2n}{p}-(n-1)}\|\Sq\vf\|_{L^2(\om_{\cp_R})}^{2/p}\|f\|_\infty^{1-\frac{2}{p}}.
\end{equation}
Now we have finished the proof of our \eqref{broadinequality2}.

\section{Appendix A: Proof of Proposition \ref{TE}}
\label{appen-trans}

For simplicity, we assume $\beta=0$.
For an $r^{-1/2}$-cap $\tau$, define
$$ \T_{\cp_\rho,\tau}:=\{ T\in\T_{\cp_\rho}:\ \textup{the~direction~of~}T\textup{~is~contained~in~}2\tau \}. $$
By the $L^2$-orthogonality, it suffices to prove for $\vg$ concentrated on wave packets from $\T_{\cp_{\rho},\tau}$, there holds 
\begin{equation}
\label{TEE}
    \int_{\cp_{r}\cap N_{r^{1/2}R^{\dc}}\bZ}|\Sq \vg|^2\lesssim R^{O(\dc)} \big( \frac{\rho}{r} \big)^{-\frac{n-m}{2}-1} \int_{\R^n}|\Sq \vg|^2+\rap(r)\|g\|_2^2.
\end{equation}
We just need to prove for each component of $\vg$ separately, that is:
\begin{equation}
\label{TEE-each}
    \int_{\cp_{r}\cap N_{r^{1/2}R^{\dc}}\bZ}|g_j|^2\lesssim R^{O(\dc)} \big( \frac{\rho}{r} \big)^{-\frac{n-m}{2}-1} \int_{\R^n}|g_j|^2+\rap(r)\|g_j\|_2^2.    
\end{equation}

\smallskip

Let us do some reductions to \eqref{TEE-each}. We first state a higher dimensional analogue of Lemma 2.1 in \cite{Wu}.
\begin{lemma}\label{localL2}
Suppose that $g\!\!:\!\ZR^n\!\!\to\!\ZC$ is a function whose Fourier support is contained in $N_{\rho^{-1}}(\Ga_j(\si))$. Let $M^2 \!\le\! r\! \le\! \rho$ and let $\cp_r\subset\ZR^n$ be a rescaled $r$-ball. Then
\begin{equation}
\label{local-l2}
    \|g\|_{L^2(\cp_r)}\lesssim r^{1/2}\rho^{-1/2}\|g\|_2.
\end{equation}
\end{lemma}
From this lemma, we see \eqref{TEE-each} boils down to
\begin{equation}
\label{TEE-reduced-1}
    \int_{\cp_{r}\cap N_{r^{1/2}R^{\dc}}\bZ}|g_j|^2\lesssim R^{O(\dc)} \big( \frac{\rho}{r} \big)^{-\frac{n-m}{2}} \int_{10\cp_r}|g_j|^2+\rap(r)\|g_j\|_2^2.   
\end{equation}
By further breaking $\cp_r$ into $\rho^{1/2}$-balls $\cp_r=\bigcup B$ (we can do this since $rM^{-1}\ge \rho^{1/2}$), it suffices to prove that for any $\rho^{1/2}$-ball $B$, one has 
\begin{equation}
\label{global}
    \int_{B\cap N_{r^{1/2}R^{\dc}}\bZ}|g_j|^2\lesssim R^{O(\dc)}\big( \frac{\rho}{r} \big)^{-\frac{n-m}{2}} \int_{10B}|g_j|^2+\rap(r)\|g_j\|_2^2.
\end{equation}

\noindent One can compare this estimate to Lemma 6.2 in \cite{Guth-II}.

To prove \eqref{global}, we need a result similar to Lemma 6.5 in \cite{Guth-II}. For any subspace $V$ in $\ZR^n$, we define a set of $\rho$-tubes $\ZT_{B,\tau,V}$ to be
\begin{equation}
    \ZT_{B,\tau,V}:=\{T\subset \cp_\rho:T\cap B\not=\varnothing,~\angle(\om(T),V)\lesssim \rho^{-1/2}R^{\dc}~\text{and}~\om(T)\subset2\tau\}.
\end{equation}

\begin{lemma}
\label{trans-equi-sub-lem}
Let $V$ be a subspace of $\ZR^{n}$. Then there exists a linear subspace $V'$ with the following properties:
\begin{itemize}
    \item[1. ]$\dim(V)+\dim(V')=n$. 
    \item[2. ]$V$ and $V'$ are quantitatively transverse, in the sense that for any unit vectors $v\in V$ and $v'\in V'$,
    \begin{equation}
        \angle(v,v')\gtrsim 1.
    \end{equation}
    \item[3. ]If $g$ has Fourier support in $N_{\rho^{-1}}\Ga_j(\tau)$, $g$ is concentrated on wave packets from $\ZT_{B,\tau,V}$, $\Pi$ is any plane parallel to $V'$ and $y\in\Pi\cap B$, then 
    \begin{equation}
    \label{transverse-equi-subspace}
        \int_{\Pi\cap B(y,r^{1/2}R^{\dc})}|g|^2\lesssim R^{O(\dc)}\Big(\frac{\rho^{1/2}}{r^{1/2}}\Big)^{-\dim(V')}\int_{\Pi\cap 10B}|g|^2+\rap(r)\|g\|_2^2.
    \end{equation}
\end{itemize}
\end{lemma}

\begin{proof}
The key is to find the linear subspace $V'$ such that the projection of supp $\wh g$ to $V'$ is contained in some $\rho^{-1/2}R^{\dc}$-ball:
\begin{equation}
    \proj_{V'}(\textup{supp}\ \wh g)\subset B_{C\rho^{-1/2}R^{\dc}}\cap V'.
\end{equation}
Then \eqref{transverse-equi-subspace} follows by the same reasoning as (6.8), (6.9) in \cite{Guth-II}.

Let us construct $V'$.
Recall the definition of $\Ga_j$ in \eqref{surface} and that the Gauss map $G_j: \Ga_j\rightarrow S^{n-1}$ maps each point $\xi\in\Ga_j$ to the normal direction of $\Ga_j$ at that point. 
Define 
$$E=G_j^{-1}(V\cap\ZS^{n-1})\cap\Ga_j(\tau).$$ 
Since $g$ is concentrated on wave packets from $T_{B,\tau,V}$, we have
$$ \proj_{V'}(\textup{supp}\ \wh g)\subset N_{\rho^{-1/2}R^{\dc}}\proj_{V'}(E). $$ 
So it suffices to prove
\begin{equation}\label{sufficeto}
    \proj_{V'}(E)\subset B_{C\rho^{-1/2}}\cap V'.
\end{equation}

Suppose that $\dim(V)=m$.
Since $G_j$ is a diffeomorphism, we see $E$ is an $(m-1)$-dimensional submanifold of $\Ga_j(\tau)$. We pick any point $\xi_0\in E$ and set $x_0:=G_j(\xi_0)$. We
choose 
$$V':=(T_{\xi_0}E\oplus \R x_0)^\perp, $$
where $T_{\xi_0}E$ is the tangent space of $E$ at $\xi_0$ and $\perp$ means the orthogonal complement.

Since the diameter of $E$ is $\lesssim r^{-1/2}\le \rho^{-1/4}$, we have $E$ lies in the $C\rho^{-1/2}$-neighborhood of $T_{\xi_0}E$, so \eqref{sufficeto} holds. 
It remains to verify that $V$ and $V'$ are quantitatively
transverse. By the definition of $V'$, it's equivalent to verify that $T_{\xi_0}E\oplus \R x_0$ and $V$ are not too orthogonal in the sense that: For any $w\in T_{\xi_0}E\oplus \R x_0$, there exists a $v\in V$ such that  $\angle(w,v)<\pi/2-c$ for some $c>0$ only depending on the surface. 

Since $x_0\in V$, it suffices to prove for $w\in T_{x_0}E$.
Actually, we just choose $v=d G_j(w)$ and will show for $w\in T_{\xi_0}E$, there holds $|\langle w,d G_j(w)\rangle|\gtrsim |w||d G_j(w)|$, so $\angle(w,d G_j(w))<\pi/2-c$.

Without loss of generality, we may assume $\xi_0=0$, $x_0=(0,\cdots,0,1)$ and $$\Ga_j=\{ (\bar\xi, \Phi(\bar\xi)): |\bar\xi|\le 1/2 \},$$
for some $\Phi$ with $\nabla \Phi(\xi_0)=0$, $\nabla^2 \Phi(\xi_0)$ positively definite.
We have $$|\langle w,d G_j(w)\rangle|=|\langle \nabla^2 \Phi(\xi_0)w,w\rangle|\gtrsim |w|^2\gtrsim |w||d G_j(w)|.$$
This finishes the proof.
\end{proof}

\medskip
We can prove \eqref{global} from Lemma \ref{trans-equi-sub-lem} by following the same argument as in \cite{Guth-II} page 113-114, so we omit the details. Therefore, we proved Proposition \ref{TE}.

\section{Appendix B: Proof of Lemma \ref{lemdec}}

For simplicity, we assume $\beta=0$.
We will prove Lemma \ref{lemdec} by doing a series of reductions. In this appendix, let us assume that $\{\Ga_j\}$ are defined in \eqref{msurface}. First, we recall our lemma.

\begin{lemma}[Local, small cap]\label{dec1}
Let $\si\subset \ZS^{m-1}$ be a cap of radius $M^{-1}$, and $\Tau_{\sigma}=\{\tau\}$ be a collection of $K^{-1}M^{-1}$-caps that tile $\sigma$. Let $\vg=\{g_1,\cdots,g_R\}$ be any vector-valued function, such that each $g_j:\R^{m}\rightarrow \ZC$ has Fourier support in $N_{K^{-2}M^{-2}}\Ga_j(\si)$.
Let $B_{MK^2\times M^2K^2}\subset \R^{m}$ be a rectangle of dimensions $MK^2\times\cdots\times MK^2\times M^2K^2$, pointing to the direction $c_\si$. Then for $2 \leq p \leq 2+\frac{2}{m-1}$, we have
\begin{equation}
    \|\Sq \vg\|_{L^{p}(B_{MK^2\times M^2K^2})}\lesssim_\e (KM)^{\e}(\sum_{\tau \in \Tau_{\sigma}}\|\Sq \vg_\tau\|^2_{L^{p}(\om_{B_{MK^2\times M^2K^2}})})^{1/2}.
\end{equation}
\end{lemma}

Note that Lemma \ref{dec1} is stated in the local way for small caps. Here ``local" means that the integration domain is $B_{MK^{2}\times M^2K^2}$ (when it is $\R^m$, we call ``global"); ``small cap" means that each component of our function $\vg$ has Fourier support on caps that are determined by the small cap $\si$.

It is standard that Lemma \ref{dec1} is equivalent to the following global version.

\begin{lemma}[Global, small cap]\label{dec2}
Let $\si\subset \ZS^{m-1}$ be a cap of radius $M^{-1}$, and $\Tau_{\sigma}=\{\tau\}$ be a collection of $K^{-1}M^{-1}$-caps that tile $\sigma$. Let $\vg=\{g_1,\cdots,g_R\}$ be any vector-valued function, such that each $g_j:\R^{m}\rightarrow \ZC$ has Fourier support in $N_{K^{-2}M^{-2}}\Ga_j(\si)$. Then for $2 \leq p \leq 2+\frac{2}{m-1}$, we have
\begin{equation}
    \|\Sq \vg\|_{L^{p}(\R^m)}\lesssim_\e (KM)^{\e}(\sum_{\tau \in \Tau_{\sigma}}\|\Sq \vg_\tau\|^2_{L^{p}(\R^m)})^{1/2}.
\end{equation}
\end{lemma}

Now we drop the restriction on the support of $\wh \vg$ and prove a stronger lemma.

\begin{lemma}[Global]\label{dec3}
Let $\Tau=\{\tau\}$ be a collection of $K^{-1}$-caps that tile $\ZS^{m-1}$. Let $\vg=\{g_1,\cdots,g_R\}$ be any vector-valued function, such that each $g_j:\R^{m}\rightarrow \ZC$ has Fourier support in $N_{K^{-2}}\Ga_j$.
Then for $2 \leq p \leq 2+\frac{2}{m-1}$, we have
\begin{equation}\label{dec3ineq}
    \|\Sq \vg\|_{L^{p}(\R^m)}\lesssim_\e K^{\e}(\sum_{\tau \in \Tau}\|\Sq \vg_\tau\|^2_{L^{p}(\R^m)})^{1/2}.
\end{equation}
\end{lemma}

We can make further reductions on the shape of the hypersurfaces $\{\Ga_j\}$. The readers will see that it suffices to prove \eqref{dec3ineq} when each $\Ga_j$ is a paraboloid. We need to define a new family of hypersurfaces. Let $\{A_j\}_{j=1}^R$ be $(m-1)\times (m-1)$ symmetric matrices whose eigenvalues lie in $[1/2,2]$. Define the paraboloids 
\begin{equation}
    \label{defpa} P_j:=\{ (\bar\xi,\xi_m):\xi_m=\langle A_j\bar\xi,\bar\xi\rangle, |\bar\xi|\le 1/2 \}.
\end{equation}
If $\vg=\{g_1,\cdots,g_R\}$ satisfies supp$(\wh g_j)\subset N_{K^{-2}}P_j$. 
One can also define $\vg_\tau$ in the same way as \eqref{gtau}, but with $\{\Ga_j\}$ replaced by $\{P_j\}$.
We state the following lemma.

\begin{lemma}[Global, paraboloid]\label{dec4}
Let $\Tau=\{\tau\}$ be a collection of $K^{-1}$-caps that tile $\ZS^{m-1}$. Let $\vg=\{g_1,\cdots,g_R\}$ be any vector-valued function, such that each $g_j:\R^{m}\rightarrow \ZC$ has Fourier support in $N_{K^{-2}}P_j$. 
We have
\begin{equation}\label{dec4ineq}
    \|\Sq \vg\|_{L^{p}(\R^m)}\lesssim_\e K^{\e}(\sum_{\tau \in \Tau}\|\Sq \vg_\tau\|^2_{L^{p}(\R^m)})^{1/2}.
\end{equation}
\end{lemma}

We show Lemma \ref{dec4} implies Lemma \ref{dec3}. The trick can be found, for example, in Chapter 12 of \cite{demeter2020fourier}.
\begin{proof}[Sketch proof of Lemma \ref{dec4} implying Lemma \ref{dec3}]
Let $\Dec(K,\{\Ga_j\})$ be the best constant so that \eqref{dec3ineq} holds. Let $\Dec(K)$ be the best constant so that \eqref{dec4ineq} holds. By Lemma \ref{dec4}, we have $\Dec(K)\lesssim_\e K^\e$. Our goal is to prove $\Dec(K,\{\Ga_j\})\lesssim_\e K^\e$.

First, we choose a set of $K^{-1/3}$-caps $\{\al\}$ that tile $\ZS^{m-1}$. For a fixed cap $\al$ among them, we consider the truncated hypersurfaces $\{ \Ga_j(\alpha) \}_{j=1}^R$. Let $\xi_{j,\al}=G_j^{-1}(c_\al)\in \Ga_j$ be the point where the normal direction of $\Ga_j$ is $c_\al$. Let $L_\al$ be the $(m-1)\times (m-1)$ matrix  ($L_\al$ does not depend on $j$) such that $\{\xi: \xi_n=L_\al\bar\xi\}$ is parallel to the tangent space of $\Ga_j$ at $\xi_{j,\al}$.
Define $A_{j,\al}:=D^2\Phi_j(\bar \xi_{j,\al})$ (recall $\Phi_j$ in \eqref{msurface}), then $A_{j,\al}$ is symmetric with all eigenvalues lying in $[1/2,2]$.
We have the Taylor's expansion of $\Phi_j$ at $\xi_{j,\al}$:
\begin{equation}
\nonumber
    \Phi_j(\bar\xi_{j,\al}+\bar\xi)=(\xi_{j,\al})_n+L_\al\bar\xi+ \langle A_{j,\al}\bar\xi,\bar\xi \rangle+O(K^{-1}),\ \ \textup{for~}|\bar\xi|\le K^{-1/3}.
\end{equation}

Now we define our paraboloids:
\begin{equation}\label{defpar}
P_{j,\al}:=\{ \xi: \xi_m=\langle A_{j,\al}\bar\xi,\bar\xi \rangle,\ |\bar\xi|\le K^{-1/3} \}. 
\end{equation}
Here is an important observation: there is an affine map (which is the so-called parabolic rescaling) $\bar\xi\rightarrow A_\al\bar\xi+b_{j,\al}$ that maps $N_{K^{-1}}\Ga_j(\al)$ to $N_{CK^{-1}}P_{j,\al}$.

Since $\{P_{j,\al}\}_{j=1}^R$ are paraboloids truncated in a region of radius $K^{-1/3}$, we can easily find a rescaling: $\bar\xi\rightarrow K^{1/3}\bar\xi,\ \xi_n\rightarrow K^{2/3}\xi_n$, so that after performing such rescaling,
$\{P_{j,\al}\}_{j=1}^R$ become the paraboloids truncated in the region $\{|\bar\xi|\le 1/2 \}$, who satisfy the condition as in \eqref{defpa}.

By the definition of $\Dec(K^{2/3})$, on one hand we have
\begin{equation}
    \|\Sq \vg_{\al}\|_{L^{p}(\R^m)}\lesssim \Dec(K^{2/3}) (\sum_{\tau \subset \al}\|\Sq \vg_\tau\|^2_{L^{p}(\R^m)})^{1/2}.
\end{equation}
On the other hand, by the definition of $\Dec(K^{1/3},\{\Ga_j\})$, we have
\begin{equation}
    \|\Sq \vg\|_{L^{p}(\R^m)}\lesssim \Dec(K^{1/3},\{\Ga_j\}) (\sum_{\al}\|\Sq \vg_\al\|^2_{L^{p}(\R^m)})^{1/2}.
\end{equation}
As a result, we have
\begin{equation}
    \|\Sq \vg\|_{L^{p}(\R^m)}\lesssim \Dec(K^{1/3},\{\Ga_j\})\cdot\Dec(K^{2/3}) (\sum_{\tau\in\Tau}\|\Sq \vg_\tau\|^2_{L^{p}(\R^m)})^{1/2},
\end{equation}
which implies
\begin{equation}
\nonumber
    \Dec(K,\{\Ga_j\})\lesssim\Dec(K^{1/3},\{\Ga_j\})\cdot\Dec(K^{2/3}).
\end{equation}
Since $\Dec(K)\lesssim_\e K^\e$, by bootstrapping, we can obtain
$\Dec(K,\{\Ga_j\})\lesssim_\e K^\e$.
\end{proof}

Lemma \ref{dec4} is equivalent to the following local version.  Let us slightly change our notation to make it consistent with other references. We replace $K$ by $R^{1/2}$; we assume that we have $J$ paraboloids $\{P_j\}_{j=1}^J$ and write $\vg=\{g_1,\cdots,g_J\}$. We will see in our proof that the implicit constant does not depend on $J$.
\begin{lemma}[Local, paraboloid]\label{dec5}
Let $\Tau=\{\tau\}$ be a collection of $R^{-1/2}$-caps that tile $\ZS^{m-1}$. Let $\vg=\{g_1,\cdots,g_J\}$ be any vector-valued function, such that each $g_j:\R^{m}\rightarrow \ZC$ has Fourier support in $N_{R^{-1}}P_j$. 
We have
\begin{equation}\label{dec5ineq}
    \|\Sq \vg\|_{L^{p}(B_{R})}\lesssim_\e R^{\e}(\sum_{\tau \in \Tau}\|\Sq \vg_\tau\|^2_{L^{p}(\om_{B_{R}})})^{1/2}.
\end{equation}
\end{lemma}

\begin{remark}

\rm

This paraboloid version allows us to do rescaling much easier. 
\end{remark}

In the rest of Appendix B, we prove Lemma \ref{dec5}.
We adopt the notation as in \eqref{setofcap}. For a set of caps $\Si=\{\si\}$, we define
\begin{align}
    \vg_{\Si}&(x):= \sum_{\si \in\Si} \vg_{\si}(x),\\
    \Sq\vg_{\Si}(x):=& \big( \sum_{j=1}^J |\sum_{\si \in\Si} g_{j,\si}(x)|^2 \big)^{1/2}.
\end{align}

We use induction on the dimension $m$ and radius $R$. Suppose Lemma \ref{dec5} holds for dimension $m-1$ and radius $\leq R/2$.

Our strategy is to use the broad-narrow analysis. Narrow part will be dealt with by applying the induction. In order to deal with the broad part, we need the vector-valued multilinear restriction estimate. To state the estimate, let us introduce some definition: we call a collection of sets $A_1,\ldots,A_m \subset \ZS^{m-1}$ $\nu$-transverse if
\begin{equation}
    |a_1 \wedge a_2 \wedge \ldots \wedge a_m| > \nu
\end{equation}
for every points $a_i \in A_i$. For convenience, we introduce the averaged integrals
\begin{equation}
\begin{split}
    &\|f\|_{L^p_{\sharp}(B_R )}^p:= \fint_{B_R}|f|^p:=\frac{1}{|B_R|}\int_{B_R}|f|^p,
    \\&
    \|f\|_{L^p_{\sharp}(\om_{B_R })}^p:=\frac{1}{|B_R|}\int|f|^pw_{B_R}.
    \end{split}
\end{equation}
Here is the vector-valued multilinear restriction estimate. This result was proved in \cite{Lee-Sqfcn} Proposition 3.10.

\begin{lemma}\label{MRT} Let $0<\nu<1$. Suppose $\Tau_1,\ldots,\Tau_{m}$ are subsets of $\Tau$ such that any choice of $\tau_1\in\Tau_1,\ldots,\tau_m\in\Tau_{m}$ are $\nu$-transverse. For $2 \leq p \leq 2m/(m-1)$, we have
\begin{equation}\label{0728.115}
\begin{split}
     \big\|\prod_{l=1}^{m} |\Sq\vg_{\Tau_l}|^{\frac{1}{m}}\big\|_{L^{p}_{\sharp}(B_R  )} 
     \lesssim_\e \nu^{-O(1)} R^{\e}
     \Big(\prod_{l=1}^m 
     \|\Sq \vg_{\Tau_l}\|_{L^{2}_{\sharp}(\om_{B_R })} \Big)^{1/m}.
\end{split}
 \end{equation}
\end{lemma}

Let us use the above lemma to finish the proof of Lemma \ref{dec5}. We introduce an intermediate scale $K\sim \log R$. Let $\Si=\{\si\}$ be a set of $K^{-1}$-caps that tile $\ZS^{m-1}$.
Somehow we abuse the notation to just write  
\begin{equation}
\vg_\si= \sum_{\tau \subset \si}\vg_\tau.
\end{equation}
Consider a partition $B_R=\cup B_{K^2}$. 
We follow the broad-narrow analysis of Bourgain-Guth \cite{Bourgain-Guth-Oscillatory}. For each $B_{K^2}$ in this partition, define the significant set $\mathcal{S}(B_{K^2})$ by 
\begin{equation}
\mathcal{S}(B_{K^2}):=\big\{ \si\in\Si: \big\|\Sq\vg_\si \big\|_{L^{p}(B_{K^2} )} \geq \frac{100}{|\Si|}
 \big\|\Sq\vg \big\|_{L^{p}(B_{K^2})}
\big\}.
\end{equation}

Suppose that there are $C K^{-m}$-transverse caps $\si_1,\ldots,\si_m$. Then by the definition of the significant set, we have
\begin{equation}
    \big\|\Sq\vg \big\|_{L^{p}(B_{K^2} )}
         \lesssim
     K^{C} \prod_{l=1}^m
     \big\|\Sq\vg_{\si_l} \big\|_{L^{p}(B_{K^2})}^{\frac1m}.
\end{equation}
After some random translations (see  pages 11--12 of \cite{MR3961084}), this is bounded by 
\begin{equation}
    \lesssim
    K^C \big\|\prod_{l=1}^{m}
         |\Sq\vg_{\si_l}|
         ^{\frac{1}{m}}\big\|_{L^{p}(B_{K^2} )}.
\end{equation}

If such transverse caps do not exist, then all the elements of $\mathcal{S}(B_{K^2})$ are contained in the $K^{-1}$-neighborhood of an $(m\!-\!1)$-dimensional subspace $V$. As in \eqref{fsi}, define
\begin{equation}
    \Si(V):=\{ \si\in\Si: \angle(\si, V)\le K^{-1} \}.
\end{equation}
In this case, we have
\begin{equation}
    \big\|\Sq\vg \big\|_{L^{p}(B_{K^2} )}
        \lesssim 
         \big\|\Sq\vg_{\Si(V)} \big\|_{L^{p}(B_{K^2} )}  + \big( \sum_{\tau\in\Tau} \big\|\Sq\vg_{{\tau}} \big\|_{L^{p}(B_{K^2} )}^p \big)^{1/p}.
\end{equation}
As a consequence, we obtain
\begin{equation}\label{116}
    \begin{split}
        \big\|\Sq\vg \big\|_{L^{p}(B_{K^2} )}
        &\lesssim K^C
         \big\|\prod_{l=1}^{m}
         |\Sq\vg_{\si_l}|
         ^{\frac{1}{m}}\big\|_{L^{p}(B_{K^2} )}\\
         &+ \big\|\Sq\vg_{\Si(V)} \big\|_{L^{p}(B_{K^2} )}+\big( \sum_{\tau\in\Tau} \big\|\Sq\vg_\tau \big\|_{L^{p}(B_{K^2} )}^p \big)^{1/p},
    \end{split} 
\end{equation}
for some $(m-1)$-dimensional subspace $V$ and $CK^{-m}$-transverse caps $\{\si_l\}_{l=1}^m$. By the  induction hypothesis on $m$, we have
\begin{equation}
    \big\|\Sq\vg_{\Si(V)} \big\|_{L^{p}(B_{K^2} )} 
         \lesssim 
         K^{\e}(\sum_{\si \in \Si}\|\Sq\vg_\si\|^2_{L^{p}(\om_{B_{K^2}})})^{1/2}.
\end{equation}
Summing over all the balls $B_{K^2} \subset B_R$ on the both sides of \eqref{116}, applying the above inequalities, we obtain
\begin{equation}
    \begin{split}
        \big\|\Sq\vg &\big\|_{L^{p}(B_R )}
        \lesssim 
         K^C
         \big\|\prod_{l=1}^{m}
         |\Sq\vg_{\si_l}|
         ^{\frac{1}{m}}\big\|_{L^{p}(B_R )}\\
         &+K^{\e}(\sum_{\si\in\Si(V)}\|\Sq\vg_\si\|^2_{L^{p}(w_{B_R})})^{1/2}+\big( \sum_{\tau\in\Tau} \big\|\Sq\vg_\tau \big\|_{L^{p}(B_{K^2} )}^2 \big)^{1/2}
         .
    \end{split} 
\end{equation}
We hope the right hand side is $\lesssim_\e R^\e( \sum_{\tau\in\Tau} \|\Sq\vg_\tau \|_{L^{p}(B_{K^2})}^2 )^{1/2}$.
For the first term on the right hand side, we apply Lemma \ref{MRT} with $\Tau_l=\{\tau\in\Tau:\tau\subset \si_l\}$ to obtain the desired estimate. For the second term, we first do parabolic rescaling so that for each $j$, $N_{R^{-1}}P_j(\si)$ (recall the definition in \eqref{defslabsurface}) becomes $N_{CK^2R^{-1}}P_j'$ for some paraboloid $P_j'$, and the slabs $\{N_{R^{-1}}P_j(\tau): \tau\in \Tau, \tau\subset\si\}$ become $KR^{-1/2}\times\cdots\times KR^{-1}\times K^2R^{-1}$-slabs of $N_{CK^2R^{-1}}P_j'$. Hence, we can apply the induction hypothesis for the radius $K^{-2}R$ to get the desired bound. We leave out the details. Finally, we proved
\begin{equation}
        \big\|\Sq\vg \big\|_{L^{p}(B_R )}\lesssim_\e 
         R^\e\big( \sum_{\tau\in\Tau} \big\|\Sq\vg_\tau \big\|_{L^{p}(B_{K^2} )}^2 \big)^{1/2}.
\end{equation}

\section{Appendix C: Proof of Corollary \ref{mixed-norm-cor-2}}
Let us assume $I=[1,2]$ without loss of generality. By the triangle inequality and the Littlewood-Paley decomposition, it suffices to show that given a smooth cutoff function $\wh\vp$ with  $\supp(\wh{\vp})\subset B^n(0,1)\setminus B^n(0,1/2)$, one has
\begin{equation}
\label{11-1}
    \Big\| \Big(\int_{R}^{2R}\big|e^{it(-\De)^{\al/2}}(\vp\ast f)\big|^q\,dt\Big)^{1/q}\Big\|_{L^p(\ZR^n)}\lesssim R^{\frac{n}{2}-\frac{n}{p}}\|f\|_p.
\end{equation}
Via a localization argument similar to Lemma 8 in \cite{Rogers}, \eqref{11-1} boils down to 
\begin{equation}
\label{11-2}
    \Big\| \Big(\int_{R}^{2R}\big|e^{it(-\De)^{\al/2}}(\vp\ast f)\big|^q\,dt\Big)^{1/q}\Big\|_{L^p(B^n_R)}\lesssim R^{\frac{n}{2}-\frac{n}{p}}\|f\|_p,
\end{equation}
with an extra assumption on $f$ that $f$ is supported in an $R$ ball in $\ZR^n$. We partition $[R,2R]$ as a union of unit intervals $\{I_k\}$, and let $\{\vp_{k}\}_k$ be a smooth partition of unity associated to it, where $\supp(\wh\vp_k)\subset [-1,1]$. Notice that $\wh {\vp\ast f}$ is supported in the unit ball. It implies that as a function of $t$, $\vp_ke^{it(-\De)^{\al/2}}(\vp\ast f)$ is essentially constant in any unit interval. Indeed, Bernstein's inequality gives 
\begin{equation}
    \Big(\int \big|\vp_ke^{it(-\De)^{\al/2}}(\vp\ast f)\big|^q\,dt\Big)^{1/q}\lesssim\Big(\int \big|\vp_ke^{it(-\De)^{\al/2}}(\vp\ast f)\big|^2\,dt\Big)^{1/2}.
\end{equation}
Summing up all $\vp_k$ and using the imbedding $l^2 \hookrightarrow l^q$ so that
\begin{align}
\nonumber
    \Big\| \Big(\int_{R}^{2R}\!\!\big|e^{it(-\De)^{\al/2}}(\vp\ast f)\big|^qdt\Big)^\frac{1}{q}\Big\|_{L^p(B^n_R)}\!\lesssim & \,\Big\| \Big(\int_{R/2}^{3R}\big|e^{it(-\De)^{\al/2}}(\vp\ast f)\big|^2dt\Big)^\frac{1}{2}\Big\|_{L^p(B^n_{R})}\\ \nonumber
    &+\rap(R)\|f\|_p.
\end{align}
Apply a rescaled version of  \eqref{local-smoothing-2} to the estimate above to obtain \eqref{11-2}. \qed

\bibliographystyle{alpha}
\bibliography{bibli}

\end{document}